%% file: main.tex
\documentclass[]{article}
\title{Sub-Riemannian mean curvature flow for image processing} 

\author{G. Citti\thanks{Dipartimento di Matematica, Universit\'{a} di Bologna. ({giovanna.citti@unibo.it})} \and B. Franceschiello\thanks{Center of Mathematics, CNRS-EHESS, Paris ({benedetta.franceschiello@ehess.fr})} \and  G. Sanguinetti\thanks{Department of Mathematics and Computer Science, TUE. ({g.r.sanguinetti@tue.nl})}\and A. Sarti\thanks{Center of Mathematics, CNRS-EHESS, Paris ({alessandro.sarti@ehess.fr})}}

\usepackage{amsfonts,amsmath,amsthm,amssymb}

\theoremstyle{plain}
\newtheorem{theorem}{Theorem}[section]
\newtheorem{lemma}[theorem]{Lemma}
\newtheorem{corollary}[theorem]{Corollary}
\newtheorem{proposition}[theorem]{Proposition}

\newtheorem{definition}[theorem]{Definition}

\usepackage{fullpage}
\newcommand{\R}{\mathbb R}

\usepackage{graphics,graphicx,color}
\usepackage[colorlinks=true,citecolor=blue]{hyperref}

\begin{document}
\maketitle

\begin{abstract}
In this paper we reconsider the sub-Riemannian cortical model of image completion introduced in \cite{A5, A4}. 
This model combines two mechanisms, the sub-Riemannian diffusion and the concentration, giving rise to a diffusion driven motion by curvature. In this paper we give a formal proof of the existence of viscosity solutions of the sub-Riemannian motion by curvature. Furthermore we illustrate the sub-Riemannian finite difference scheme used to implement the model and we discuss some properties of the algorithm. Finally results of completion and enhancement on a number of natural images are shown and compared with other models.
\end{abstract}


\markboth{G. Citti et al.}{Sub-Riemannian mean curvature flow }

\section{Introduction}

\input{introductionandsubriemannian}

\section{Existence of viscosity solutions}\label{sec3}

\input{existence}

\section{Numerical scheme}\label{sec4}

\input{numerical}

\section{Results}\label{sec5}

\input{resultsandconclusions}

\section{Acknowledgements}
The research leading to these results has received funding from the People Programme (Marie
Curie Actions) of the European Union's Seventh Framework Programme FP7/2007-2013/ under REA
grant agreement n607643. Author GS has received
funding from the European Research Council under the ECs 7th Framework
Programme (FP7/2007 2014)/ERC grant agreement No. 335555.


\end{document}

%% file: introductionandsubriemannian.tex
In this work we aim to face the problem of completion (also known as inpaiting) and contours enhancement: both require an analysis of elongated structures, such as level lines or contours. The word \textit{inpainting}, first coined by Bertalmio, Sapiro, Caselles, and Ballester in \cite{Bertalmio}, refers to a technique whose purpose is to restore damaged portions of an image. In terms of psychology of vision inpainting corresponds to \textit{amodal completion}, a phenomenon described by Kanizsa in \cite{Kanizsa}, which consists in the completion of the occluded part of an image. This problem can be faced with many techniques such as geometric and variational instruments which start from the selection of existing boundaries or algorithms acting directly on the image space, based on the Gestalt law of good continuation and the classical findings of psychology of vision. This last class of algorithms goes back to the works 
of Kanizsa \cite{Kanizsa}, Ullman \cite{Ullm}, Horn \cite{horn} and find the prototype of models for curve completion into Mumford's Elastica functional (see \cite{K1}): 
$$\int_\gamma \sqrt{1 + k^2} \mbox{\,}ds.$$
This expression minimizes a functional defined on the curve itself, where $k$ is the curvature along the curve $\gamma$, parametrized by arc length $s$. 
The models proposed by Mumford, Nitzberg and Shiota, see \cite{K1} and \cite{Mumford}, and by Masnou and Morel, see \cite{Morel}, generalize this principle to the level sets of a gray-valued image. In fact, if $D$ is a squared defined in $\mathbb{R}^2$ and $I:D\rightarrow \mathbb{R}$ is an image, the considered functionals are second order functionals defined on $I$:
$$
\int_D |\nabla I|\bigg( 1+ \bigg|div \bigg(\frac{\nabla I}{|\nabla I|}\bigg) \bigg|^2\bigg)\,dxdy.
$$
A model expressed by a system of two equations, one responsible for boundary extraction, and one for figure completion was proposed by Bertalmio, Sapiro, Caselles, and Ballester in \cite{Bertalmio}. 
In Sarti et al \cite{SMS} the role of the observer was considered, letting evolve by curvature in the Riemannian metric associated to the image a fixed surface called \textit{point of view surface}. 

Close to the completion problem there is the one of enhancement, allowing to make the structures of images more visible, while reducing the noise. Also for contours enhancement, models expressed as non linear PDE dependent on the gradient were proposed in the same years. The first models were due to Nitzberg
and Shiota \cite{NS1992}, Cottet and Germain \cite{CG1993}. Later on Weickert \cite{W1998, W1999} proposed a coherence-enhancing diffusion method, called CED. Let us also mention the models of Tschumperl\'e \cite{T2006}, who included curvature into the diffusion process in order to improve enhancement.

\bigskip

A different class of algorithms to perform both completion and enhancement is directly inspired by the functionality of the visual cortex, 
which is expressed in terms of contact structures and Lie groups with a sub-Riemannian metric. We refer to Section \ref{sec2} for a precise definition of this metric and for references of the mathematical aspects of the problem. 
We only recall here that PDEs in this setting are totally degenerate at every point, so that classical results on PDE 
cannot be applied to study their solutions. Then we will follow the approach provided by Evans and Spruck in \cite{A6}.  The first geometric models of the functionality of the visual cortex date back to the papers of Hoffmann \cite{A20}, Koenderink \cite{Koenderink}, Zucker \cite{A23}. 
Petitot and Tondut in \cite{A21} proposed a model of single boundaries completion through constraint minimization in a contact structures, obtaining a neural counterpart of the models of Mumford, see \cite{K1}. Sarti and Citti in \cite{A4} and \cite{A5} proposed to model the cortical structure as 
Lie group $SE(2)$ of rigid transformations of the plane with a sub-Riemannian metric. 
Each 2D retinotopic image is lifted to a surface in this higher dimensional structure and processed by the long range via a diffusion driven motion by curvature. This model performs amodal completion on the whole image and can be considered the cortical counterpart of the models of Mumford, Nithberg and Shiota \cite{Mumford}.
  A large class of algorithms of image processing in Lie groups were developed by M. Van Almsick, R. and M. Duits and B.M. Ter Haar Romeny in \cite{ALMDUIT}. The main instrument he developed is an invertible map from the 2D retinal image to the feature space expressed in terms of Fourier instruments. He faced both the problem of contour completion and enhancement in presence of 
crossing and bifurcating lines. Let us also recall the results of Boscain \cite{Boscain},  who implemented the diffusion step of the model introduced in \cite{A4} using instruments of Group Fourier transform, similar to the ones proposed by Duits in \cite{DFI}.

\bigskip
The paper is organized as in the following. In Section \ref{sec2} we reconsider the completion model of Citti and Sarti (see \cite{A4}) and propose to extend its application to perform contours enhancement. Indeed we will see that the action of the sub-Riemannian diffusion is strongly directional, 
and it reduces noise only in the direction of contours. Vice versa it performs no diffusion in the orthogonal direction, preserving the structures of the images.

Since the model is formally expressed as a diffusion driven 
motion by curvature, in Section \ref{sec3} we consider the mean curvature equation in the $SE(2)$ group and prove the existence of viscosity solutions. Existence results for this type of equations in the Euclidean setting are well known (see for example \cite{A6})  but in the sub-Riemannian setting the existence of solution was known only for the Heisenberg group (see \cite{A3}, \cite{A26}, \cite{A3}). 
Here we start filling the gap, extending the existence results of \cite{A6} in the present setting. In particular we look for sub-Riemannian viscosity solution as limit of the 
analogous problem in an approximating Riemannian setting. The main difficult we have to face is the fact that in the sub-Riemannian setting the 
derivatives are replaced by directional derivatives, which in general do not commute.  Hence we will have to introduce a new basis of the tangent space which commutes with the natural one.

In section  \ref{sec4} we discuss  the numerical scheme used to implement the mean curvature flow. Finally in section \ref{sec5} we conclude the paper showing some numerical results of completion  and enhancement, and comparing with different methods.

\section{Sub-Riemannian operators in image processing}\label{sec2}
In this section we first recall the cortical model of image completion proposed by Citti and Sarti in \cite{A4} and \cite{A5}. By simplicity we will focus only on the 
image processing aspects of the problem, neglecting all the cortical ones. 
Then we show that this mechanism
can be adapted to perform contour enhancement.

\subsection{Processing of an image in a sub-Riemannian structure}

\subsubsection{Lifting of the level lines of an image} \label{limage}

The cortical based model of completion, proposed by Citti and Sarti in \cite{A4} and \cite{A5}, lifts each level line of a 2D image  $I(x,y)$ in the retinal plane to a new curve in the group $SE(2)=\mathbb{R}^2 \times \textit{S}^1$ of rigid motions of the Euclidean plane, used to model the functional architecture of the primary visual cortex (V1). Precisely if a level line of $I$ is represented as a parametrized 2-dimensional curve $\gamma_{2D}=(x(t),y(t))$ and the vector  $\vec{X}_1=(\cos(\bar \theta(t)), \sin(\bar \theta(t)))$ is the unitary tangent to the curve $\gamma_{2D}$ at the point $t$, we say that  $\bar \theta(t)$ is the orientation of $\gamma_{2D}$ at the point $t$. Then the  choice of this orientation lifts the 2D curve $\gamma_{2D}$ to a new curve: \begin{equation}(x(t),y(t)) \rightarrow (x(t),y(t),\bar \theta(t)) \in \mathbb{R}^2 \times \mathit{S}^1 \label{liftingeq} \end{equation}
as it is shown in figure \ref{liftfigure1}. By construction the tangent vector to any lifted curve $\gamma_{3D}=(x(t),y(t),\bar \theta(t)) $ in $SE(2)$ can be represented as a linear combination of the vector fields: $$X_1 =\cos(\theta) \partial_x + \sin(\theta) \partial_y$$
$$X_2=\partial_\theta.$$
In other words we have associated to every point of  $R^2 \times S^1$ the vector space spanned by vectors $X_1$ and $X_2$.

\begin{figure}
\centering
\includegraphics[width=.4\textwidth]{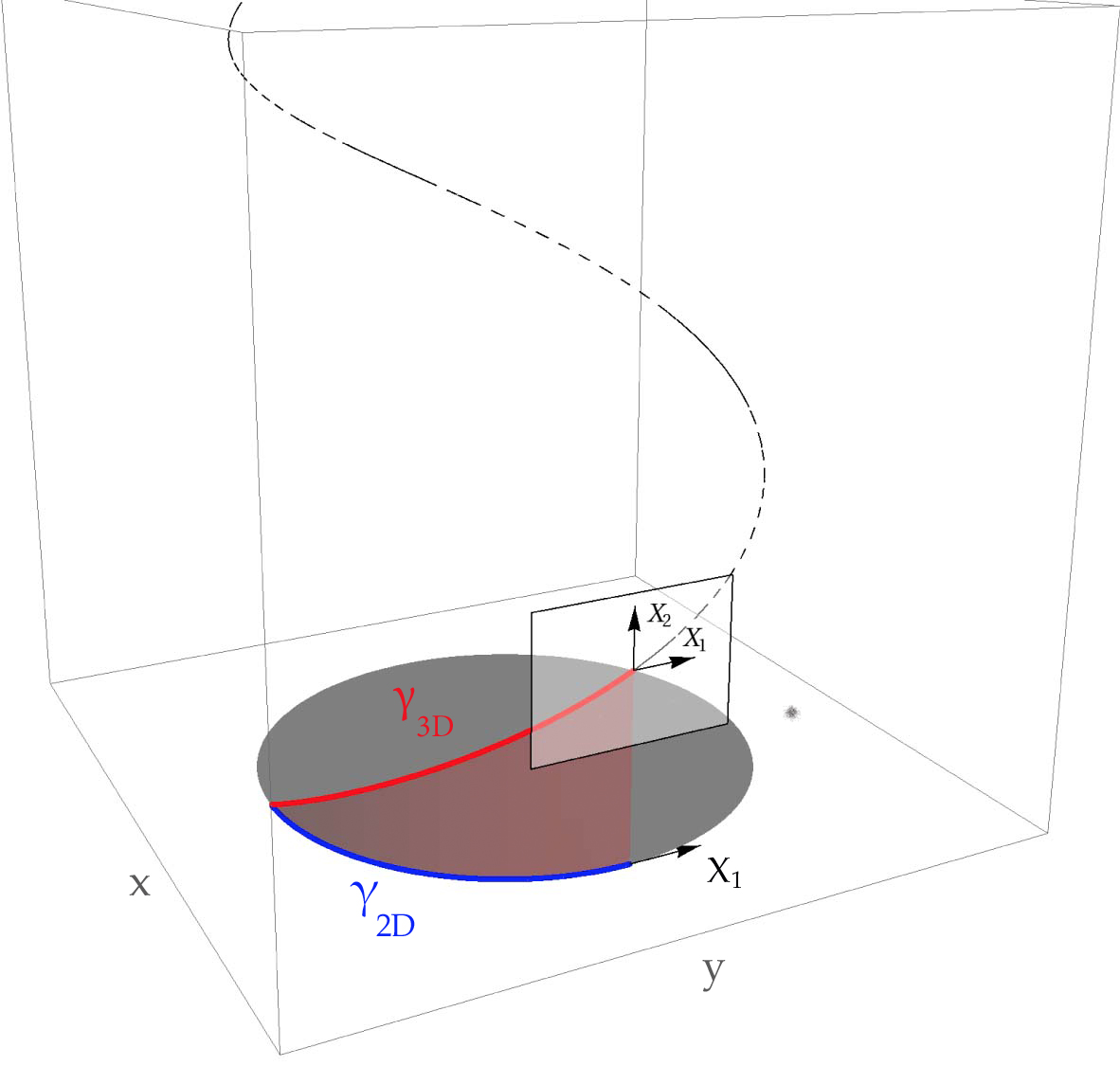} \label{liftfigure1}
\caption{A contour represented by the curve $\gamma_{2D}(t)$ is lifted into the roto-translation group obtaining the red curve $\gamma_{3D}(t)$. The tangent space of the roto-translation group is spanned by the vectors $X_1$ and $X_2$.}
\end{figure}

\subsubsection{The sub-Riemannian structure}

 We will call horizontal plane and denote it as $HM$ the tangent plane generated by $X_1$ and $X_2$  at every point.  
Let us also note that there are no lifted curve with a non-vanishing component in the orthogonal direction $$X_3=  -\sin(\theta) \partial_x + \cos(\theta)\partial_y.$$ 
However derivations in the direction $X_3$ can be recovered by means of the commutator:
 $$X_3= X_1 X_2 - X_2 X_1 = [X_1,X_2]= -\sin(\theta) \partial_x + \cos(\theta)\partial_y.$$
This condition ensures that  $X_1$, $X_2$ and their commutators of any order span the Euclidean tangent space of $SE(2)$ at every given point, i.e. they satisfy the H\"{o}rmander condition, see \cite{Horm2}. Then the structure obtained with this lifting process is sub-Riemannian. On the plane $HM$ we define the metric $g_0$ which makes $X_1$ and $X_2$  orthonormal. Hence, if a vector $a= a_1 X_1 + a_2 X_2 \in HM$, its horizontal norm is:
\begin{equation}
\lvert a \rvert_0 = \sqrt{(a_1)^2 + (a_2)^2 }.
\end{equation}
The first classical properties of the distance in these spaces have been established by Nagel, Stein and Wainger (see \cite{Nagel}), and Gromov (see  \cite{Gromov}).  We refer to Hladky (see \cite{A8}) and the references therein for recent contributions. 

In this setting the vector fields play the same role as derivatives in the standard setting. 
Hence we will say that a function $u: \mathbb{R}^2 \times \mathit{S}^1 \rightarrow \mathbb{R}$ is of class $C^1$ in the sub-Riemannian sense (we will denote it as $u\in C^1_{SR}$) if there exists $X_1 u$ and $X_2 u$ and they are continuous. In this case we will call horizontal gradient $\nabla _0$:
$$\nabla_0 u = (X_1 u) X_1+ (X_2 u) X_2.$$
From the definition stated before it follows that the norm of the horizontal gradient is: \begin{equation} |\nabla _0 u| = \sqrt{(X_1u)^2 + (X_2u)^2 }.\end{equation}
In other words the horizontal gradient is the projection of the standard Euclidean gradient of $u$ on the horizontal plane $HM$.

\subsubsection{Lifting of the image to a regular surface}\label{liftimage}

Since each level line of the image $I$ is lifted to a curve in the 3D cortical space, 
the whole image is lifted to a graph $$(x,y)\rightarrow (x,y, \bar \theta(x,y)).$$
Clearly the graph of $\theta$ can be interpreted as the zero level set of the function $u$
$$u(x,y,\theta)= \theta - \bar \theta(x,y),$$
and it can be identified as a regular surface in the 
sub-Riemannian structure. 
The notion of regular  surface $S$ was first introduced by Franchi, Serapioni and Serracassano in \cite{A24}: 
\begin{equation}
S = \{(x,y,\theta): u(x,y,\theta)=0 \mbox{\, and \,} \nabla_0 u(x,y,\theta) \not =0\}.
\end{equation}
Since the sub-Riemannian surface $S$ is union of horizontal curves, we say that it is foliated in horizontal curves. The horizontal normal of $S$ is defined as $$\nu_0 = \frac{\nabla_0 u}{|\nabla_0 u|} .$$
Note that in 
a smooth surface there can be points where the Riemannian gradient is not $0$, but its projection on the $HM$ plane vanishes:
$$\nabla_0 u =0.$$ 
Points which have this property are called \textit{characteristics} and the normal is not defined at them. However these points are not present in lifted surfaces.

\subsubsection{Diffusion and concentration algorithm}

We have seen in subsection \ref{liftimage}  how to lift an image $I(x,y)$, to a surface $S$.
After that let us lift the level lines of the image $I(x,y)$ to the function $$v(x,y,\bar\theta(x,y))=I(x,y)$$
defined on the surface. The surface $S$ and the function $v$ defined on $S$  will be processed through differential operators defined on $SE(2)$, which model the propagation of information in the cortex. More precisely two mechanisms operate on the lifted surface $S$: 
\begin{enumerate}
\item[(a)] a sub-Riemmanian diffusion along the vector fields $X_1$ and $X_2$ which model the propagation of information through the cortical lateral connectivity. This operator can be expressed as 
$$\partial_t - X_1^2 - X_2 ^2$$
where $X_1^2$ expresses the second derivative in the direction $X_1$. 
The operator is formally degenerated, in the sense that its second fundamental form has $0$ determinant at every point. It has been deeply studied starting from the classical works of H\"ormander in \cite{Hormy},  Rothshild and Stein in \cite{Roth} and  Jerison \cite{Jerison} and it is known that it is hypo-elliptic. After that  a large literature has been produced on these type of operators, and we refer to \cite{Cap2} for recent presentation of the state of the art.
\item[(b)] a concentration on the surface of maxima to model the non-maximal suppression mechanism and the orientation tuning.
\end{enumerate} 
In the Euclidean setting Merrimann, Bence and Osher, proved in \cite{Merriman} the convergence of a similar two step algorithm to the motion by curvature. 
In Citti and Sarti \cite{A4} and \cite{A5} the authors studied the motion when (a) and (b) are 
applied iteratively and proved that at each step the surface performs an increment in the normal direction 
with speed equal to the sub-Riemannian mean curvature. 

\subsubsection{Mean curvature flow}

The notion of curvature of a $\mathit{C}^2$ surface at non characteristic points  is already well understood, see (\cite{Garofalo}, \cite{Hladky}, \cite{Cheng}, \cite{Ritore}, \cite{Cap2}). It can be defined either as first variation of the area functional, either as limit of the mean curvature of the Riemannian approximation or as horizontal divergence of the horizontal normal: 
$$K_0= \mbox{div}_0 (\nu_0) = \mbox{div}_0 \bigg(\frac{\nabla_0 u}{|\nabla_0 u|} \bigg).  $$
If each point of the surface evolves in the direction of the normal vector with speed equal to the mean curvature, we say that the surface is evolving by mean curvature. From the previously expression of the curvature we formally get the following equation for the flow, which we can call horizontal (or sub-Riemannian) mean curvature flow:
\begin{equation}\label{Smcf}
\left\{\begin{array}{lll}
u_t = \sum\limits_{i,j=1}^{2} \left( \delta_{i,j} - \frac{X^0_iuX^0_ju}{|\nabla_0 u|^2} \right) X^0_iX^0_ju & \mbox{ in } & \Omega \subset \mathbb{R}^2\times S^1 \\
u(\cdot ,0)=u_0 & &
\end{array}\right.
\end{equation}
where $\delta_{ij}$ is the Kronecker function. 
An existence result for this equation was not known. We will provide in next section an existence theorem which will allow us to handle also characteristic points, 
and is expressed in terms of viscosity solutions.

\subsubsection{Laplace-Beltrami flow}

Citti and Sarti also conjectured that as a result of the previous mechanisms the function $v(x,y,\bar\theta(x,y))$, which contains the gray-levels values, evolves through the flow described by the Laplace Beltrami operator $\Delta_{LB}$: 
   \begin{equation}\label{Smcl}
             \left\{\begin{array}{lll}
              v_t = \sum\limits_{i,j=1}^{2} \left( \delta_{i,j} - \frac{X^0_iuX^0_ju}{|\nabla_0 u|^2} \right)X^0_iX^0_jv & \mbox{ in } & \Omega \subset \mathbb{R}^2\times S^1  \\
              v(\cdot ,0)=v_0. & &
             \end{array}\right.
           \end{equation}
From now on in order to simplify notations we will denote: 
\begin{equation}\label{aij}A_{ij}^{0}(\nabla_0 u) =   \delta_{i,j} - \frac{X^0_iuX^0_ju}{|\nabla_0 u|^2} , \quad i,j=1, 2.\end{equation}
Let us note that the described equations become degenerate and the solutions are regular only along the directions of the foliation.

\subsection{Enhancement and Inpainting in Sub-Riemannian geometry} \label{section algorithm}

\subsubsection{Inpainting of missing parts of the image} 

In the previous section we described an algorithm proposed in \cite{A4} for restoring damaged portions of an image, where the corrupted set $\omega$ is known a priori.  

\begin{itemize}
\item[1]{ An image $I(x,y)$ are lifted to a surface $S = \{(x, y, \bar \theta(x,y)\}$ in the Lie group $SE(2)$ of rotation and translation, and the gray level of the image $I(x,y)$ to a function $v(x,y,\bar\theta(x,y))$ defined on $S$. In the lifting the corrupted part of the image becomes $\Omega= \omega\times S^1$, where no surface is defined.}
\item[2]{The surface  $S$ and $v(x,y,\bar\theta(x,y))$ are processed via the algorithm of diffusion and concentration in the corrupted region $\Omega$, where we impose Dirichlet boundary conditions. This leads to the motion by mean curvature of the surface $S$ and to a Laplace Beltrami flow for $v(x,y,\bar\theta(x,y))$.} 
\item[3]{The final result is obtained by re-projecting onto the plane of the image the value of the intensity $v(x,y,\bar\theta(x,y))$.}
\end{itemize}
The algorithm has been implemented in \cite{GCS} via a diffusion and concentration method, 
while it has been implemented via the curvature equation in \cite{A4}.

\subsubsection{Enhancement of boundaries} 

One of the scope of this paper is to extend the previous completion algorithm to solve the problem of contours enhancement. The aim of this technique is to provide a regularization in the direction of the boundaries, making them clearer and brighter and eliminating noise.  We refer to the paper of Duits  \cite{DFI},\cite{DFII} for some results of image enhancement in this space. Precisely he lifts the image $I$ on the 3D features space, using an invertible map defined through 
Fourier analysis. The lifted version of the image $I$ is processed in the 3D  space 
and then reprojected on the 2D plane to recover an enhanced version of the image $I$. In particular he also provides results of enhancement in presence of bifurcation or crossing. 
In this paper we face the same problem adapting the algorithm recalled in the previous section. 
\begin{itemize}
\item[1]{First we lift the level lines of an image $I(x,y)$ to a surface $S = \{(x, y, \bar \theta(x,y))\}$  and we lift the gray levels of $I(x,y)$ to a function $v(x, y, \bar \theta(x,y))$ always defined on $S$. } 
\item[2]{Then we process the surface $S$ via a mean curvature flow and $v$ via a Laplace-Beltrami flow. In order to perform enhancement we propose here to let equations (\ref{Smcf}) and (\ref{Smcl}) evolve in the full domain $\R^2 \times S^1$. Let us remark that lifting the image in the 3D group allows to solve the problem of crossing elongated structures.
Indeed if two lines cross in the 2D space and have different orientations, they are lifted to the 3D space to two different planes, allowing completion and enhancement. The directional diffusion will give place to a regularization only in the direction of contours.} 
\item[3]{Finally we project into the plane of the image the value of the gray intensity $v(x,y,\theta(x,y))$.}
\end{itemize}

%% file: existence.tex
In this section we provide the main result of this paper which is the proof of existence of viscosity solutions for the mean curvature flow in $SE(2)$. 
We explicitly note that we do not need to develop new results for the Laplace-Beltrami operator, which is linear.

As we can immediately observe the PDE becomes degenerate in the singularities of the horizontal gradient of the solution $u(\ldotp,t)$.
Furthermore, just as in the Euclidean space, we cannot expect the smoothness of the solution to be preserved for all times. 
The notion of viscosity solution has been introduced in order to overcome these type of problems. 
The method of generalized (viscosity) solutions independently developed by Chen, by Giga and Goto \cite{Visc[4]}, by Evans and Spruck
\cite{A6}, by Crandall, Ishii and Lions \cite{Visc[8]} is now applied to large 
classes of degenerate equations \cite{Visc[19]}. Recently it has been extended also in the sub-Riemannian setting, see \cite{W2} \cite{A3} \cite{A26}. 
Finally let us recall that Dirr et al. \cite{A25} have recently studied a probabilistic approach to the mean curvature flow in the context of the Heisenberg group.

\subsection{The notion of viscosity solution}
\subsubsection{Viscosity solutions in Jet spaces}
The cortical model previously discussed, associates to each 2D curve $\gamma_{2D}$ its orientation. This procedure can be considered as a lifting of the initial image $I(x,y)$ to a 
new function $u$ in the Jet-space $\mathbb{R}^2 \times \mathit{S}^1$ of position and orientation. We refer to Petitot and Tondut, who first described the analogous cortical process as a lifting in a jet space \cite{A21}.
In this section, we will sometimes denote $\xi =(x,y,\theta)$ as an element of $\mathbb{R}^2 \times \mathit{S}^1$. 
It is then natural to lift the function $u$ into another Jet-space which contains the formal analogous  of its sub-Riemannian gradient $\nabla_0u$  and the formal analogous  of its 
second derivatives $X^0_iX^0_j$. The definition of viscosity solution in Jet-spaces is based on the Taylor expansion, expressed in terms of these differential objects. Since the analogous of the increment in the direction of the gradient $p$ is expressed through the exponential map, then the increment from a point $\xi$ in the direction $\sum_{i=1}^2\eta_i X^0_i$ is expressed as
$$u\Big(\exp(\sum_{i=1}^2\eta_i X^0_i)(\xi),t+s\Big) - u(\xi,t).$$
At non regular points, such as kinks,  there is not either a unique vector $p$ which identifies the horizontal gradient and a unique matrix $r_{ij}$ which identifies the horizontal Hessian. Hence we need to give a more general notion: a couple $(p,r)$ where  $p_i, i=1,2$ denotes an horizontal vector and $(r_{ij})$ a $2 \times 2$ matrix is a superjet $\mathcal{J}^+$ for $u$ if it satisfies the following formal analogous of the Taylor development:
\begin{equation}
u\Big(exp(\sum_{i=1}^2\eta_i X_i)(\xi),t+s\Big)- u(\xi,t) 
\leq \sum_{i=1}^{2}p_i\eta_i + \frac12 \sum_{i,j=1}^2r_{ij}\eta_i\eta_j + qs + o(|\eta|^2+s^2).
\end{equation}
Let us note that if the superject exists it can be used in place of the derivative; furthermore a function $u$ is a Jet-space viscosity solution if the differential equation in which the derivatives are replaced with the elements of the superjet is satisfied. More precisely: 
\begin{definition}
A function $u \in \mathit{C}(\mathbb{R}^2 \times \mathit{S}^1 \times [0,\infty)) \cap \mathcal{L}^{\infty}(\mathbb{R}^2 \times \mathit{S}^1 \times [0,\infty))$ is a jet space-viscosity subsolution of equation (\ref{Smcf}) if for every $(p_H, r_{ij})$ in the super-Jet we have:

\begin{equation}
q \leq 
\left\{\}\begin{array}{ll}
\sum_{i,j=1}^{2} A^0_{ij}(p_H) r_{ij} & \mbox{ if }|p_H| \neq 0 \\
\sum_{i,j=1}^{2} A^0_{ij}(\tilde p) r_{ij} & \mbox{ for some }|\tilde p|\leq 1, \mbox{if } |p_H|=0. 
\end{array}\right.
\end{equation}
\end{definition}
An analogous definition is provided for a viscosity supersolution. Then a viscosity solution is a function which is both a subsolution and a supersolution.

\subsubsection{Viscosity solutions via test functions}
The definition of viscosity solution in Jet-space of a second order equation can be identified as the approximation of the solution $u$ via a second order polynomial, whose coefficients are exactly the elements $(p_i, r_{ij})$ of the Jet space. More generally we assume that the function $u$ is locally approximate by a smooth test function $\phi$. This definition imposes the behavior of the function $u$ at points where $u-\phi$ attains a maximum. At such points $u$ and $\phi$ will have the same first derivatives, so that $\nabla_0 \phi$ results to be an exact evaluation of the approximation of $\nabla_0 u$. Looking at second derivatives, it follows that for every $i$ we have: $$X_i X_i(u-\phi)\leq 0,$$
so that the curvature of $\phi$ is an upper bound for the curvature of $u$. Due to this observations we can give the following definition:

\begin{definition}
A function $u \in \mathit{C}(\mathbb{R}^2 \times \mathit{S}^1 \times [0,\infty))$ is a  viscosity subsolution of (\ref{Smcf}) in $\mathbb{R}^2 \times \mathit{S}^1 \times [0,\infty)$ if for any $(\xi,t)$ in $\mathbb{R}^2 \times \mathit{S}^1 \times [0,\infty)$ and any function $\phi \in \mathit{C}(\mathbb{R}^2 \times \mathit{S}^1 \times [0,\infty))$ such that $u- \phi$ has a local maximum at $(\xi,t)$ it satisfies:
\begin{equation}
\partial_t \phi \leq 
\left\{\begin{array}{ll} 
\sum_{i,j=1}^{2}A^0_{ij}(\nabla_0 \phi)X_iX_j \phi \mbox{,}& \mbox{ if }  |\nabla_0 \phi| \neq 0 \\
\sum_{i,j=1}^{2}A^0_{ij}(\tilde p)X_iX_j\phi , & \mbox{ for some }\tilde  p \in \mathbb{R}^2 \mbox{, } |\tilde p| \neq 1 \mbox{, if} |\nabla_0 \phi| = 0  
\end{array}\right.	  			         
\end{equation}
A function $u \in \mathit{C}(\mathbb{R}^2 \times \mathit{S}^1 \times [0,\infty))$ is a viscosity supersolution of (\ref{Smcf}) if:
\begin{equation}
\partial_t \phi \geq 
\left\{\begin{array}{ll}  \sum_{i,j=1}^{2}A^0_{ij}(\nabla_0 \phi)X_iX_j \phi & \mbox{if \,\,}  |\nabla_0 \phi| \neq 0 \\
			  \sum_{i,j=1}^{2}A^0_{ij}(\tilde p)X_iX_j\phi & \mbox{for some\,\,} \tilde p \in \mathbb{R}^2 \mbox{,\,} |\tilde p| \neq 1 \mbox{, if\,} |\nabla_0 \phi| = 0  \\
\end{array}\right.			         
\end{equation}
\end{definition}
\begin{definition}
A \textit{viscosity solution} of (\ref{Smcf}) is a function $u$ which is both a viscosity subsolution and a viscosity supersolution.
\end{definition}
\begin{theorem}
The two definitions of jet spaces viscosity solution and viscosity solution are equivalent.
\end{theorem}

\subsubsection{Vanishing viscosity solutions}
A vanishing viscosity solution is the limit of the solutions of approximating regular problems. 

Let us first explicitly note that the coefficients $A_{ij}$ are degenerate: when the gradient vanishes, they are not defined. Hence we will apply the regularization procedure proposed by Evans and Spruck in \cite{A6} to face singularities, which consists in replacing the coefficients with the following ones:
$$A_{ij}^{\tau}(p)= \left(\delta_{ij} -\frac{p_ip_j}{|p|^2+\tau} \right).$$
This approximation has a clear geometric interpretation, already provided by Evans and Spruck. In equation (\ref{Smcf}) each level set of $u$ evolves by mean curvature. What we obtain adding a new parameter is the evolution of the graph of $u$ 
$$\Gamma_t^{\tau}=\{ (\xi,\xi_{n+1}) \in \mathbb{R}^{n+1}|\xi_{n+1}=u(\xi, t) \}$$
and the introduction in the space of a metric depending on $\tau$.
In this approximation equation (\ref{Smcf}) reads as: 
\begin{equation}\label{mcfgraph}
\left\{\begin{array}{lll}
u_t = \sum\limits_{i,j=1}^{2} A^\tau_{ij}(\nabla_0 u) X_iX_ju & \mbox{ in } & \Omega \subset \mathbb{R}^2\times S^1 \\
u(\cdot ,0)=u_0. & &
\end{array}\right.
\end{equation}

We will now introduce a Riemannian approximation of the mean curvature flow in the graph approximation we made before.
We extend $g_0$ on the whole space $SE(2)$ to a metric $g_\epsilon $ which makes the vectors $X_1$, $X_2$, $\epsilon X_3$ orthonormal. Let us note that $g_\epsilon $ is the Riemannian completion of the horizontal metric. From now on, in order to simplify notations we will always denote  
\begin{equation}\label{campie}X_1^\epsilon = X_1, \quad X_2^\epsilon = X_2, \quad X_3^\epsilon = \epsilon X_3.\end{equation} 
The Riemannian gradient associated to the metric $g_\epsilon$ will be represented as:
$$\nabla_\epsilon u = X_1^\epsilon u X_1^\epsilon + X_2^\epsilon u X^\epsilon_2 + X_3^\epsilon u X^\epsilon_3 $$
and, using the fact that $X^\epsilon$  are orthonormal, we get:
\begin{equation}
\lvert \nabla _\epsilon u\rvert = \sqrt{(X_1u)^2 + (X_2u)^2 + \epsilon^2 (X_3 u)^2}.
\end{equation}
In the Riemannian setting equation (\ref{mcfgraph}) reads as: 
\begin{equation}\label{mcf}
\left\{\begin{array}{lll}
u_t = \sum\limits_{i,j=1}^{3} A^{\epsilon, \tau}_{ij}(\nabla_\epsilon u) X^\epsilon_iX^\epsilon_ju & \text{ in } & \Omega \subset \mathbb{R}^2\times S^1 \\
u(\cdot ,0)=u_0 & &
\end{array}\right.
\end{equation}
where 
$$A^{\epsilon, \tau}_{ij}(\nabla_\epsilon u)=\left( \delta_{i,j} - \frac{X^\epsilon_iuX^\epsilon_ju}{|\nabla_\epsilon u|^2+ \tau } \right).$$

In order to prove the existence of a solution we apply another regularization, always introduced by Evans and Spruck. It consists in adding a Laplacian, ensuring that the matrix of the coefficients have strictly positive smallest eigenvalue. Then the approximated coefficients will be: 
$$A_{ij}^{\epsilon,\tau,\sigma}(p)=A_{ij}^{\epsilon,\tau}(p)+\sigma\delta_{ij}$$
and the associated equation becomes:
\begin{equation}\label{mcfest}
\left\{\begin{array}{lll}
u_t = \sum\limits_{i,j=1}^{3} A_{ij}^{\epsilon,\tau,\sigma}(\nabla^\epsilon u)  X^\epsilon_iX^\epsilon_ju & \text{ in } & \Omega \subset \mathbb{R}^2\times S^1 \\
u(\cdot ,0)=u_0. & &
\end{array}\right.
\end{equation}
This condition makes the coefficients satisfy the coercivity condition and allows to apply the standard theory of uniformly parabolic equations.

We are now in condition to give a third definition of vanishing viscosity solution, 

\begin{definition} 
A function $u$ is a vanishing viscosity solution of (\ref{Smcf}) if it is limit of a sequence of solutions 
$u^{\epsilon_k,\tau_k,\sigma_k}$ of equation (\ref{mcfest}).
\end{definition}

We will see that any Lipschitz continuous vanishing viscosity solution is a viscosity solution of the same equation 
(see Theorem \ref{vivi} below). 

\subsection{Solution of the approximating equations}
The aim of this sub-section is to  prove the existence of solutions for the approximating equation (\ref{mcfest}) and, even more important,  to establish estimates independent of all parameters, which hold true also the limit equation (\ref{Smcf}). 
We will study solutions on the whole space, since this is the assumption needed for the enhancing algorithm.
\begin{theorem} 
\label{First part of existence} 
Assume that $u_0 \in \mathit{C}^{\infty}(\mathbb{R}^2 \times \mathit{S}^1)$ and that it is constant on the exterior of a cylinder, i.e. there exists $S>0$ such that: 
\begin{equation}\label{constantS}
u_0 \mbox{ is constant on \quad} \mathbb{R}^2 \times \mathit{S}^1   \cap \{x^2 + y^2  \geq S\}.
\end{equation}
Then there exists a unique solution $u^{\epsilon,\tau, \sigma} \in \mathit{C}^{2,\alpha}(\mathbb{R}^2 \times \mathit{S}^1 \times [0,\infty))$ of the initial value problem (\ref{mcfest}). 
Moreover, for all $t>0$ one has: 
\begin{equation}
\lVert u^{\epsilon, \tau, \sigma}(\cdot, t) \rVert_{\mathcal{L}^{\infty}(\mathbb{R}^2 \times \mathit{S}^1)}
\leq \lVert u_0 \rVert_{\mathcal{L}^{\infty}(\mathbb{R}^2 \times \mathit{S}^1)} \label{stima0}
\end{equation}
\begin{equation} \label{stima1}
\lVert  {\nabla}_{E}u^{\epsilon,\tau, \sigma}(\cdot, t)\rVert_{\mathcal{L}^{\infty}(\mathbb{R}^2 \times \mathit{S}^1)}
\leq \lVert  {\nabla}_{E} u_0 \rVert_{\mathcal{L}^{\infty}(\mathbb{R}^2 \times \mathit{S}^1)}. \end{equation}
where $\nabla_{E}(\cdot)$ denotes the Euclidean gradient. \end{theorem}

Since the previous estimates do not depend on $\sigma$ and $\epsilon$, they will be stable 
when these parameters go to $0$: 
\begin{corollary}
The solution $u^\tau$ of the equation (\ref{mcfgraph}) satisfies conditions (\ref{stima0}) and (\ref{stima1}). 
\end{corollary}

\bigskip
This results generalizes to $SE(2)$ the previous results of \cite{A6} and \cite{A3}. 
The main difficult to face in the extension is the fact that the vector fields $X^\epsilon$ does not commute, hence it is not easy to find a nice equation satisfied by the gradient. Following the approach proposed by Mumford in \cite{K1}, we will take the derivatives along the direction of green a family of vector fields $\{Y_i\}_{i=1,\ldots,3}$, which are right invariant with respect to the group law. It is a general fact that these vector fields commute with the left invariant ones. We recall this result in the special case of our vector fields, for reader convenience:

\begin{lemma}\label{Vector fields commute}
 A right invariant basis of the tangent space can 
be defined as follow:
$$\begin{array}{l}
Y_1=\partial_x \\ 
Y_2=X_2 + (x\cos\theta+y\sin\theta)X_3+(x\sin\theta - y\cos\theta)X_1 = \partial_\theta - y \partial_x + x \partial_y\\ 
Y_3=\partial_y.\end{array}
$$
We will see that the vector fields $\{X^\epsilon_i\}_{i=1,2,3}$ defined in (\ref{campie}) commute with $\{Y_i\}_{i=1,2,3}$.
\end{lemma}
\begin{proof}
We will calculate their Lie bracket:
\begin{eqnarray*}
[X_1,Y_1]&=&(\cos\theta \partial_x + \sin\theta \partial_y)\partial_x-\partial_x(\cos\theta \partial_x + \sin\theta \partial_y)\\ & = & \cos\theta \partial_{xx}+\sin\theta \partial_{yx} -\cos\theta \partial_{xx}-\sin\theta \partial_{xy} = 0. 
\end{eqnarray*}

Since the coefficients of $Y_2$ do not depend on $\theta$ it is clear that 
\begin{eqnarray*}
[X_2,Y_2]&=0.
\end{eqnarray*}
Finally 
\begin{eqnarray*}
[X_3,Y_3]&=&(\sin\theta \partial_x - \cos\theta \partial_y)\partial_y-\partial_y(\sin\theta \partial_x - \cos\theta \partial_y)\\&=& \sin\theta \partial_{xy}-\cos\theta \partial_{yy} -\sin\theta \partial_{xy}+\cos\theta \partial_{yy} = 0
\end{eqnarray*}
The other combinations commute analogously.
\end{proof}

The first step of the proof of Theorem \ref{First part of existence} is the existence of the function $u^{}$ and its $L^\infty$ bound: 

\begin{theorem} \label{linfti} Under the assumption of Theorem \ref{First part of existence} on the initial datum, the initial value problem (\ref{mcfest}) has an unique solution $u^{\epsilon,\tau, \sigma} \in \mathit{C}^{2,\alpha}(\mathbb{R}^2 \times \mathit{S}^1 \times [0,\infty))$  
such that 
\begin{equation}
\lVert u^{\epsilon, \tau, \sigma}(\cdot, t) \rVert_{\mathcal{L}^{\infty}(\mathbb{R}^2 \times \mathit{S}^1)}
\leq \lVert u_0 \rVert_{\mathcal{L}^{\infty}(\mathbb{R}^2 \times \mathit{S}^1)} 
\end{equation}
\end{theorem}

\begin{proof}
 For $\sigma>0$, consider the problem associated to equation (\ref{mcfest}) on a cylinder  
$B(0,r) \times [0,T]$, 
with initial data 
\begin{equation}
u_r^{\epsilon,\tau,\sigma}(\cdot,0)=u_0, 
\end{equation}
and constant value on the lateral boundary of the cylinder. Note that coefficients ${A_{ij}^{\epsilon,\tau \sigma}}$ satisfy the uniform parabolic condition:
\begin{equation}
\sigma|p|^2 \leq A_{ij}^{\epsilon, \tau, \sigma}(\tilde p)p_ip_j  \label{Coercivity condition}
\end{equation}
for each $\tilde p, p \in \mathbb{R}^3$. hence the theory of parabolic equations on bounded cylinders ensures that for every fixed value of the parameters there exists an  unique smooth solution $u^{\epsilon,\tau,\sigma}_r$ (see for example Ladyzenskaja, Solonnikov, Ural'tseva \cite{A12}). By maximum principle we have 
\begin{equation}
 \lVert u_r^{\epsilon, \tau, \sigma}(\cdot,t)\rVert_{\mathcal{L}^{\infty}(\mathbb{R}^2 \times \mathit{S}^1)} \leq \lVert u_0\rVert_{\mathcal{L}^{\infty}( \mathbb{R}^2 \times \mathit{S}^1)}. \label{estimate0}
\end{equation}
Letting $r$ tend to $\infty$, we obtain a solution $u^{\epsilon,\tau,\sigma}$ defined on the whole $\mathbb{R}^n \times [0,T]$ such that $$\lVert u^{\epsilon,\tau,\sigma} \rVert_{\infty} \leq \lVert u_0\rVert_{\mathcal{L}^{\infty}( \mathbb{R}^2 \times \mathit{S}^1)}. $$
\end{proof}

We can now complete the second part of the proof of Theorem \ref{First part of existence} which involves the estimate of the gradient:
\begin{theorem}\label{stimagrad}
 Under the assumption of Theorem \ref{First part of existence} and \ref{linfti}, the solution of the initial value problem (\ref{mcfest}) satisfies 
\begin{equation}
\lVert  {\nabla}_{E}u^{\epsilon,\tau, \sigma}(\cdot, t) \rVert_{\mathcal{L}^{\infty}(\mathbb{R}^2 \times \mathit{S}^1)}
\leq \lVert  {\nabla}_{E} u_0 \rVert_{\mathcal{L}^{\infty}(\mathbb{R}^2 \times \mathit{S}^1)}. 
\end{equation}
\end{theorem}

\begin{proof}
From Theorem \ref{linfti} we known that there exists an unique smooth solution $u^{\epsilon, \tau, \sigma}$ of equation (\ref{mcfest}) and we only have to estimate its gradient. To this end, we can differentiate equation (\ref{mcfest}) along the directions $\{Y_i\}_{i=1,2,3}$, and using Lemma \ref{Vector fields commute}, we obtain the following equation for 
 $w_i=Y_iu^{\epsilon, \tau, \sigma}$, for all $i=1,2,3$, and for $\omega_4= Y_2 u^{\epsilon, \tau, \sigma}  - (y_0 \partial_x - x_0 \partial_y)u^{\epsilon, \tau, \sigma}$:
\begin{equation}
\frac{\partial}{\partial t}w_i=\sum_{i,j=1}^{3}\left( A_{i,j}^{\epsilon,\tau,\sigma}(\nabla_\epsilon u^{\epsilon, \tau, \sigma})X_i^{\epsilon}X_j^{\epsilon}w_i + (\partial_{\xi_k}A_{i,j}^{\epsilon,\tau,\sigma})(\nabla_\epsilon u^{\epsilon, \tau, \sigma})X_i^{\epsilon}X_j^{\epsilon}u^{\epsilon, \tau, \sigma}X_k w_i\right). \label{Differentiate2}
\end{equation}
The parabolic maximum principle applied to the previous equation yields:
\begin{equation}
\lVert Y_i u^{\epsilon, \tau, \sigma}(\cdot,t) \rVert_{\mathcal{L}^{\infty}(\mathbb{R}^2 \times \mathit{S}^1)}\leq \lVert Y_i u_0 \rVert_{\mathcal{L}^{\infty}(\mathbb{R}^2 \times \mathit{S}^1)}  \label{estimate1}
\end{equation}
This implies that 
 $$\lVert \partial_x u^{\epsilon, \tau, \sigma}(\cdot,t) \rVert_{\mathcal{L}^{\infty}(\mathbb{R}^2 \times \mathit{S}^1)}
+ \lVert \partial_y u^{\epsilon, \tau, \sigma}(\cdot,t) \rVert_{\mathcal{L}^{\infty}(\mathbb{R}^2 \times \mathit{S}^1)}
\leq C\lVert \nabla_E u_0 \rVert_{\mathcal{L}^{\infty}(\mathbb{R}^2 \times \mathit{S}^1)}.$$
We have now to establish the estimate of the derivative $\partial_\theta$. 
For every fixed value of $(x_0, y_0)$ we have 
$$|\partial_\theta u^{^{\epsilon, \tau, \sigma}} (x_0, y_0, \theta)| \leq \max_{|y-y_0|^2+ |x_0-x|^2\leq 1} |Y_2 u^{^{\epsilon, \tau, \sigma}}  - (y_0 \partial_x - x_0 \partial_y)u^{^{\epsilon, \tau, \sigma}} |\leq $$$$
\leq \max_{|y-y_0|^2+ |x_0-x|^2\leq 1} |Y_2 u_0  - (y_0 \partial_x - x_0 \partial_y)u_0 |\leq $$
$$
\leq \max_{|y-y_0|^2+ |x_0-x|^2\leq 1} |\partial_\theta u_0|  + \max_{|y-y_0|^2+ |x_0-x|^2\leq 1} |y_0-y| |\partial_x u_0| + \max_{|y-y_0|^2+ |x_0-x|^2\leq 1} |x_0-x| |\partial_y u_0 |\leq $$$$\leq\lVert  {\nabla}_{E} u_0 \rVert_{\mathcal{L}^{\infty}(\mathbb{R}^2 \times \mathit{S}^1)}$$
\end{proof} 

Let us conclude this section remarking that the proof of Theorem \ref{First part of existence} is a direct consequence of the two Theorems \ref{linfti} and \ref{stimagrad}. 

\subsection{Existence result}

In order to extend to our setting Evans and Spruck's argument in the proof of \cite{A6}, as well as the proof of \cite{A3} we need to let go to $0$ the three approximating parameters $\sigma \rightarrow 0$,  $\tau \rightarrow 0$ and $\epsilon \rightarrow 0$. 
Since the estimates we have established are uniform in all parameters, we immediately have the 
existence of a vanishing viscosity solution:

\begin{theorem}\label{vanishvisc}
Assume that $u_0 \in \mathit{C}(\mathbb{R}^2 \times \mathit{S}^1)$ is Lipschitz continuous and satisfies (\ref{constantS}). Then there exists a vanishing viscosity solution $u \in \mathit{C}^{1,0}$ of (\ref{Smcf}), which satisfies the following properties:
\begin{equation}
\lVert u(\cdot, t) \rVert_{\mathcal{L}^{\infty}(\mathbb{R}^2 \times \mathit{S}^1)}
\leq C\lVert u_0 \rVert_{\mathcal{L}^{\infty}(\mathbb{R}^2 \times \mathit{S}^1)}  
\end{equation}
\begin{equation}
\lVert {\nabla}_{E} u(\cdot, t) \rVert_{\mathcal{L}^{\infty}(\mathbb{R}^2 \times \mathit{S}^1)}
\leq C\lVert  {\nabla}_{E} u_0 \rVert_{\mathcal{L}^{\infty}(\mathbb{R}^2 \times \mathit{S}^1)}. 
\end{equation}
\end{theorem}

\begin{proof}
Since $u_0$ is constant at infinity, we immediately deduce from Weiestrass theorem that the Euclidean gradient $\nabla_E u_0$ is bounded. Employing estimates (\ref{stima0}),(\ref{stima1}) 
and Ascoli Arzel\`{a} Theorem we can extract two sequences $\{\sigma_k\}, \{\epsilon_k\},\{\tau_k\} \rightarrow 0$ of positive numbers such that $\frac{\epsilon_k}{\tau_k}\rightarrow 0$ and such that the corresponding solutions $\{u^k=u^{\epsilon_k,\tau_k, \sigma_k}\}_{k \in \mathbb{N}}$ 
are convergent in the space of Lipshitz functions. Then, by definition the limit is a Lipshitz continuous vanishing viscosity solution.
\end{proof}

We will now prove that this vanishing viscosity solution is indeed a viscosity solution:
\begin{theorem}\label{vivi}
Assume that $u_0 \in \mathit{C}(\mathbb{R}^2 \times \mathit{S}^1)$ is continuous and satisfies (\ref{constantS}). Then the vanishing viscosity solution detected in Theorem \ref{vanishvisc} is a viscosity solution $u \in \mathit{C}^{1,0}$ of (\ref{Smcf}).
\end{theorem} 

\begin{proof}
 In order to prove that $u$ is a viscosity solution we consider a function $\phi \in \mathit{C}^{\infty}(\mathbb{R}^2 \times \mathit{S}^1 \times [0,\infty))$ and we suppose that $u-\phi$ has a strict local maximum at a point $(\xi_0,t_0)\in \mathbb{R}^2 \times \mathit{S}^1 \times [0,\infty)$. 
Since $u$ is a Lipschitz continuous vanishing viscosity solution, it can be uniformly approximated 
by solutions $(u^k)$ of the approximating Riemannian problem (see also Theorem \ref{vivi}). 
As $u^k \rightarrow u$ uniformly near $(\xi_0,t_0)$, $u^k-\phi$ has a local maximum at a point $(\xi_k,t_k)$, with 
\begin{equation}
(\xi_k,t_k)\rightarrow (\xi_0,t_0) \mbox{\quad as \quad} k\rightarrow \infty \label{sequenza punti}
\end{equation}
Since $u^{k}$ and $\phi$ are smooth, we have 
$$\nabla_E u^k=\nabla_E \phi \mbox{\, , \,} \partial_t u^k=\partial_t \phi \mbox{\, and \,} \mathit{D}^2_E(u^k-\phi) \leq 0 \mbox{\, at \,} (\xi_k,t_k)$$
where $\mathit{D}^2_E$ is the Euclidean Hessian. 
Thus 
\begin{equation}
\partial_t \phi - \big(\delta_{ij} - \frac{X_i^{\epsilon_k}\phi X_j^{\epsilon_k}\phi}{|\nabla_{\epsilon_k} \phi|^2 + \tau_k^2} \big)X_i^{\epsilon_k}X_j^{\epsilon_k}\phi \leq 0 \mbox{\, at \,} (\xi_k,t_k)
\end{equation}
This inequality can be equivalently expressed in terms of the coefficients $A_{i,j}^{\epsilon,\tau}$ as follows. At the point $(\xi_k,t_k)$
\begin{eqnarray}
\partial_t \phi &-& A_{i,j}^{\epsilon_k,\tau_k}(\nabla_{\epsilon_k}\phi)X_i^{\epsilon_k}X_j^{\epsilon_k}\phi\\ &\leq& \partial_t u^k - A_{i,j}^{\epsilon_k,\tau_k}(\nabla_{\epsilon_k}u^k)X_i^{\epsilon_k}X_j^{\epsilon_k}(u^k+\phi-u^k)  \leq 0 \label{equation in which we pass to limit}
\end{eqnarray}
If $\nabla_0{\phi}(\xi_0,t_0) \neq 0$, also $\nabla_0{\phi}(\xi_k,t_k) \neq 0$ for sufficiently large $k$. Then letting $k\rightarrow \infty$ we obtain from  (\ref{equation in which we pass to limit}):
\begin{equation}
\partial_t \phi \leq \sum_{i,j=1}^{2}\big(\delta_{ij} - \frac{X_i\phi X_j\phi}{{|\nabla_0 \phi|}^2} \big) X_iX_j \phi \mbox{\, at \,} (\xi_0,t_0)
\label{important}\end{equation}
which implies that $u$ is a viscosity subsolution.\\
If $\nabla_0{\phi}(\xi_0,t_0) = 0$ then we set
$$\eta^{k}= \frac{\nabla_{\epsilon_k}\phi(\xi_k,t_k)}{\sqrt{|\nabla_{\epsilon_k}\phi(\xi_k,t_k)|^2+\tau_k^2}}$$
There exists $\eta \in \mathbb{R}^3$ such that $\eta^k \rightarrow \eta$. Note that 
$$|(\eta^k)_3|=\frac{\epsilon_k|X_3 \phi(\xi_k,t_k)|}{\sqrt{|\nabla_{\epsilon_k}\phi(\xi_k,t_k)|^2+\tau_k^2}}\leq \frac{(\epsilon_k/\tau_k)|X_3 \phi(\xi_k,t_k)|}{\sqrt{(\epsilon_k/\tau_k)^2\sum_{i=1}^{2}(X_i\phi(\xi_k,t_k))^2+1}}$$
Since the expression vanishes as $k \rightarrow \infty$ we have $\eta_3=0$. The PDE (\ref{equation in which we pass to limit}) now reads as:
$$\partial_t \phi(\xi_k,t_k) - \sum_{i,j=1}^{3}(\delta_{ij}-\eta_i^k \eta_j^k)X_i^{\epsilon_k}X_j^{\epsilon_k}\phi(\xi_k,t_k) \leq 0$$
so as $k \rightarrow \infty$ we obtain
\begin{equation}
\partial_t \phi(\xi_0,t_0) \leq \sum_{i,j=1}^{2}(\delta_{ij}-\eta_i \eta_j)X_iX_j\phi(\xi_0,t_0) 
\label{important2}\end{equation}
concluding the proof for the case in which $u- \phi$ has a local strict maximum at point $(\xi_0,t_0)$. If $u- \phi$ has a local maximum, but not necessarily a strict local maximum at $(\xi_0,t_0)$, we can repeat the argument above replacing $\phi(x,t)$ with $$\tilde{\phi}(\xi,t)= \phi(\xi,t)+|\xi-\xi_0|^4+(t-t_0)^4$$
again to obtain (\ref{important}),(\ref{important2}). Consequently $u$ is a weak subsolution. That $u$ is a weak supersolution follows analogously.

\end{proof}

From the above result we can only say that there is a subsequence of $u^{\epsilon,\tau, \sigma}$ which is convergent to the vanishing viscosity solution $u$. In order to prove the uniqueness of the vanishing viscosity solution, we would need the sub-Riemannian analogous of estimate established by Deckelnick and Dzuik in \cite{Deck}: 
\begin{proposition}\label{1}   There exists a constant $C>0$  independent of $\sigma, \tau$ and $\epsilon$ 
such that: 
\begin{equation}
\lVert u^{\epsilon,\tau, \sigma} - u\rVert_\infty \leq C \tau^\alpha
\end{equation}
\end{proposition} 

Letting $\epsilon$ and $\sigma$ go to $0$ we also get: 
\begin{equation}
\lVert u^{\tau}- u\rVert_\infty \leq C \tau^\alpha
\end{equation}
where $u^\tau$ is a solution of (\ref{mcfgraph}).

%% file: numerical.tex
In this part we provide the numerical approximation we used to implement the sub-Riemannian motion by curvature which performs inpainting and enhancement. Since our scheme is directly inspired by the classical one of Osher and Sethian (see \cite{A13}),   we will explain how to adapt the discretization to the sub-Riemannian setting.  
Let us recall here how the scheme of Osher and Sethian in \cite{A13} can be adapted to our sub-Riemannian setting.
The mean curvature flow (\ref{mcfgraph}) can be explicitly written as: 
\begin{equation}\label{explicitmcf}
u_t = \frac{(X_2(u))^2 \cdot  X_{11}(u) +  (X_1(u))^2\cdot X_{22}(u) - X_1(u) X_2(u)\cdot 2X_{12}(u) }{(X_1(u))^2 + (X_2(u))^2+ \tau} 
+ \frac{X_1(u)X_2(u)\cdot [X_1,X_2](u)}{(X_1(u))^2 + (X_2(u))^2+\tau}\end{equation}

This equation presents two distinct terms: the first part of the flow presents second order derivatives and corresponds to the curvature term, the second one has only first order derivatives and correspond to the metric connection. 
The solution $u(x,y,\theta,t)$ is discretized on a regular grid with points $x_i=i\Delta x,y_j=j\Delta y, \theta_k=k\Delta \theta$, with time discretization $t_s=s\Delta t$. We will denote$\mathit{D}^{+x}U(i,j,k,s)$, $\mathit{D}^{-x}U(i,j,k,s)$, $\mathit{D}^{0 x}U(i,j,k,s)$ the forward, backward and central  difference of a discrete function $U$ at point $(i,j,k,s)$ with respect to $x$, and use analogous notations for the other variables $y$ an $\theta$. In terms of these derivatives we will define the analogous differences in the direction of the vector fields $X_1$ and $X_2$. Precisely if we have discretized the direction $\theta$ with $K$ points, we will denote $\theta_k = k\pi/K$ for  $k=1, \cdots K$, and we will call 
$$\mathit{D}^{+X_1}U(i,j,k,s) = \cos\theta_k \mathit{D}^{+x}U(i,j,k,s) + \sin\theta_k \mathit{D}^{+y}U(i,j,k,s)$$
and analogously define backward and central  difference 
$\mathit{D}^{-X_1}$, $\mathit{D}^{0X_1}$ for the vector $X_1$ and for the vector $X_2$. 
Let us adapt the scheme proposed by Osher and Sethian in \cite{A13} to our case:
\begin{itemize}
\item[(i)] the first order term $X_1(u)X_2(u)\cdot [X_1,X_2](u)$ is discretized using the upwind scheme for $[X_1,X_2]=X_3$. 
Taking into account the upwind scheme for the vector field $X_3$,  the first order term is given by:
\begin{equation}
W^1 (U) = - \frac{
\max(-\sin\theta_k \mathit{D}^{0X_1}U \mathit{D}^{0X_2}U,0) \mathit{D}^{-x}U + \min(-\sin\theta_k \mathit{D}^{0X_1}U \mathit{D}^{0X_2}U,0) \mathit{D}^{+x}U }{|\mathit{D}^{0X_1}U|^2+ |\mathit{D}^{0X_1}U|^2 + \tau }+  \end{equation}
$$ - \frac{\max(\cos\theta_k \mathit{D}^{0X_1}U \mathit{D}^{0X_2}U ,0) \mathit{D}^{-y} U + \min(\cos\theta_k \mathit{D}^{0X_1}U \mathit{D}^{0X_2}U ,0) \mathit{D}^{+y}U}{|\mathit{D}^{0X_1}U|^2+ |\mathit{D}^{0X_1}U|^2 + \tau}.$$
\item[(ii)]  second order derivatives are implemented as 
usual as $D^{- X_1} D^{+ X_1}$, $D^{- X_2} D^{+ X_2}$, $D^{0 X_1} D^{0 X_2}$, which leads to 
second order central finite difference. We will implemented as central differences the first derivatives coefficients of $D^{- X_1} D^{+ X_1}$, $D^{- X_2} D^{+ X_2}$. The first derivative with respect to $X_1$, coefficient of the second mixed derivative, will be upwinded as before. Generalizing an idea of  \cite{POSS}, the denominator will be a mean of central derivatives: 
$$|D_{int} U|^2(i,j,k,s) + \tau  =  \frac{1}{3} \sum_{{k_1 \in \{k-1, k, k+1\}}}|D^{0 X_1} U|^2(i,j,k_1,s) +\frac{1}{5} \sum_{{i_1 \in I}}|D^{0 X_2} U|^2(i_1,j_1,k,s) + \tau,$$
where $I$ is the family of indices $I=\{(i-1, j), (i, j), (i+1, j), (i, j-1),(i, j+1) \}$.

The second order discretized operator will be denoted $W^2 (U)(i,j,k, s) $. 

\end{itemize} 

The difference equation associated to the continuous equation (\ref{explicitmcf}) will be expressed as: 
$$U(i,j,k, s+1) = U(i,j,k, s) + \Delta t (W^2 U)(i,j,k, s) + \Delta t (W^1 U)(i,j,k, s)$$
with initial condition $U(., 0) =U_0$.  
We recall that  convergence of difference schemes for the
mean curvature 
ow inspired by the scheme of Osher and Sethian has been object of a large number of
papers in the Euclidean setting. The stability of one of them was proved in \cite{POSS}.
Another monotone scheme was proposed by Crandall and Lions (see \cite{Crandall}) and its convergence was proved by Deckelnick in \cite{DECKELNICK} and Deckelnick \& Dzuik in \cite{Deck}.  

The ideas at the basis of the stability proof of \cite{POSS} can be extended to the 
present version of the Osher and Sethian scheme, leading to the following result: 

\begin{theorem}
The difference scheme presented above is stable in the sense that 
if $\Delta t \leq \frac{h^2}{10}$, then 
$$||U||_{\infty} \leq ||U_0||_{\infty} $$
\end{theorem}
\begin{proof}
If $U$ is a solution of the discrete equation, also 
$V = U - ||U_0||_{\infty} $ is a solution of the same equation: 
$$V(i,j,k, s+1) = V(i,j,k, s) + \Delta t(W^2 V)(i,j,k, s) + \Delta t(W^1  V)(i,j,k, s).$$
Hence $V(0)\leq 0$, and we have to prove $V \leq 0$, for all time. 
In order to study the term $D^{w X_3} V$ we have to discuss the sign of 
$a_1=-\sin\theta_k \mathit{D}^{0X_1}V \mathit{D}^{0X_2}V$ and 
$a_2=\cos\theta_k \mathit{D}^{0X_1}V \mathit{D}^{0X_2}V$: we will assume that they are 
both positive since the proof is similar in all the other cases: 
In this case 
$$(W^1 V)(i,j,k, s) =
-
\frac{  a_1 (V(i,j,k, s)- V(i-1,j,k, s)) + a_2 (V(i,j,k, s)- V(i,j-1,k, s)}{|\mathit{D}^{0X_1}V|^2+ |\mathit{D}^{0X_2}V|^2 + \tau }
 \leq $$
$$-\frac{ (\cos(\theta_k)- \sin(\theta_k)) \mathit{D}^{0X_1}V(i,j,k,s) \mathit{D}^{0X_2}V(i,j,k,s)}{(|\mathit{D}^{0X_1}V|^2+ |\mathit{D}^{0X_2}V|^2 + \tau )h^2} V(i,j,k, s)\leq-\frac{ V(i,j,k, s)}{2 h^2}$$

Analogously, having upwinded the coefficient of $(W^2 V)(i,j,k, s)$, e get a similar behavior. The mixed derivatives term can be estimated as: 
$$
- \frac{2\cos\theta_k \mathit{D}^{0X_1} (\mathit{D}^{0X_2})V  \mathit{D}^{0X_2}V }{|D_{int} V|^2 +\tau} V(i,j,k, s)\leq - \frac{10 V(i,j,k, s)}{h^2}.$$

In conclusion 
$$V(i,j,k, s+1)\leq V(i,j,k, s)(1 - \frac{10 \Delta t}{h^2})\leq 0.$$
The assertion then follows by induction.
\end{proof}

Now we recall that the equation is uniformly parabolic in sub-elliptic sense. Arguing as in 
\cite{Deck} the estimates of 4th order derivatives can be reduced to the estimates of graphs 
over the considered group. Hence estimates can be obtained by a recent result of Capogna, Citti and Manfredini (see \cite{CapCitManf}). 
Since for $\tau$ fixed the equation is uniformly parabolic in sub-elliptic sense, these estimates, 
allow to prove that 

\begin{theorem}
If $u^\tau$ is the solution of (\ref{mcfgraph}) with initial condition $u_0$ and $U$ is the solution of the 
discrete scheme considered here, and $\alpha$ is fixed, there exist a constant $C= C(\tau, h, \alpha)$ such that if $\Delta t\leq C(\tau, h)$, then 
$$|u^\tau(i \Delta x, j \Delta y, k \Delta \theta, s \Delta t) - U(i,j,k, s) |\leq \tau^\alpha,$$
\end{theorem}

As a consequence, applying the uniqueness Theorem \ref{1}, we deduce  
the following convergence result  for the  solution of the mean 
curvature equation (\ref{Smcf}) with initial condition $u_0$ 
$$|u(i \Delta x, j \Delta y, k \Delta \theta, s \Delta t) - U(i,j,k, s) |\leq\tau^\alpha$$
as $\Delta t\leq C(\tau, h)$. 

%% file: resultsandconclusions.tex
In this section 
we compare the results of the model with previous works, and discuss our results. 
The completion method, first proposed in \cite{A5}, has now a solid theoretical background, 
and is compared with more recent results obtained by different authors. 
Results of our new model enhancement are proposed, 
comparing with previous results of Duits \cite{DFII}. 
Finally we show an example of inpainting and enhancement of an image.

\subsection{Inpainting results}

Since the lifting procedure is based on the selection of the orientation of level lines at every point, the algorithm performs particularity well for completing gray level images which have non vanishing gradient at every point. Hence we will start with a simple artificial image of this type. The 
algorithm performs very well for completion of curved level lines.
\begin{figure}
\centering
\includegraphics[width=0.3\textwidth]{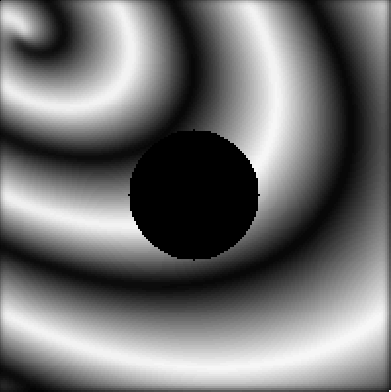}\hspace{.02\textwidth}
\includegraphics[width=0.3\textwidth]{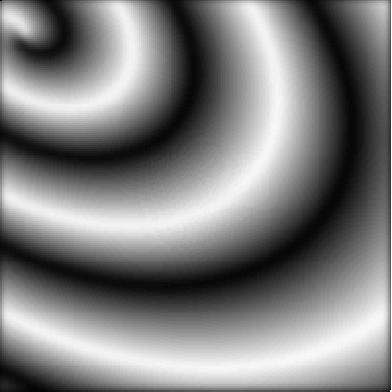}\hspace{.02\textwidth}
\caption{An example of completion performed by the algorithm. In this artificial image the image gradient is lifted in the $\mathbb{R}^2 \times S^1$ space and the black hole is completed by mean curvature flow. Since the level lines of the image are approximately circular, the algorithm performs very well.}
\end{figure}

We will now test the algorithm on natural images. In the first image a black hole is present, 
and the algorithm correctly reconstructs the missed part of the image:
\begin{figure}
\centering
\includegraphics[width=.6\textwidth]{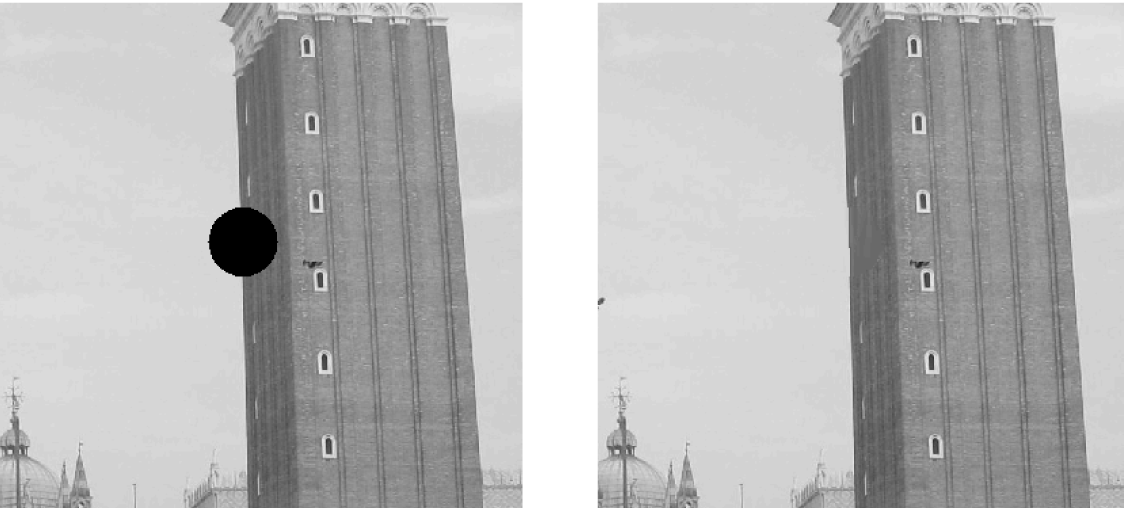}\hspace{.2\textwidth}\\
\caption{Completion result on a real image through sub-Riemannian mean curvature flow in $\mathbb{R}^2 \times S^1$, as described in the paper.} 
\end{figure}

The model (\cite{A4}) studied here  performs 
completion via the curvature flow. Very recently Boscain et al. in \cite{Boscain}, 
tried to replace this non linear equation by simple diffusion. 
In figure \ref{fig:isokernels1} left we consider an image courtesy of U. Boscain \cite{Boscain},  
partially occluded by a grid  and show the results of completion performed in \cite{Boscain} (second image from left), by the heat equation on the 2D space (third image from left) and by the Citti and Sarti model (right). 
A detail is shown in figure \ref{detail1}.
Since the considered image is a painting, extremely smooth, with low contrast, 
the 2D heat equation is already able to perform a simple version of completion. 
In this case the implementation of sub-Riemannian diffusion \cite{Boscain} 
provides a worse result , while the curvature model reconstructs correctly 
the missed contours and level lines.

\begin{figure}
\centering
\includegraphics[width=.22\textwidth]{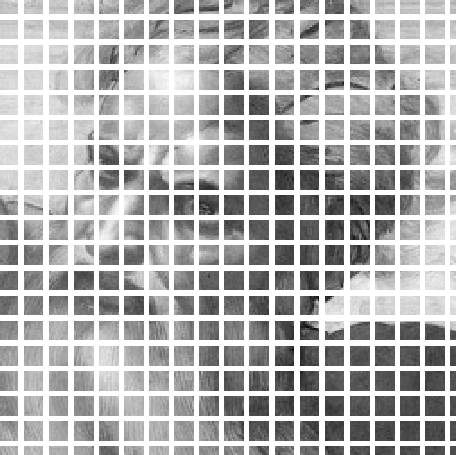}\hspace{.02\textwidth}
\includegraphics[width=.22\textwidth]{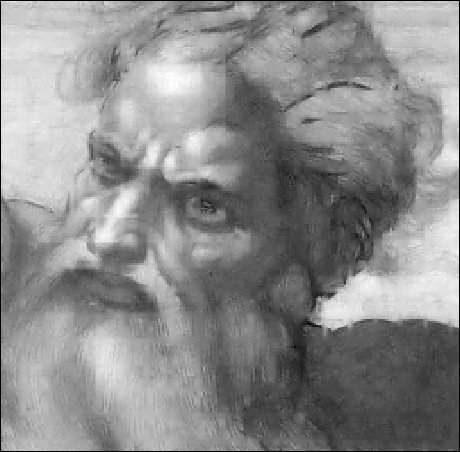}\hspace{.02\textwidth}
\includegraphics[width=.22\textwidth]{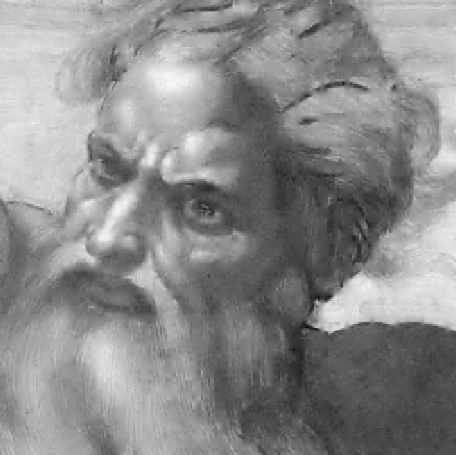}\hspace{.02\textwidth}
\includegraphics[width=.22\textwidth]{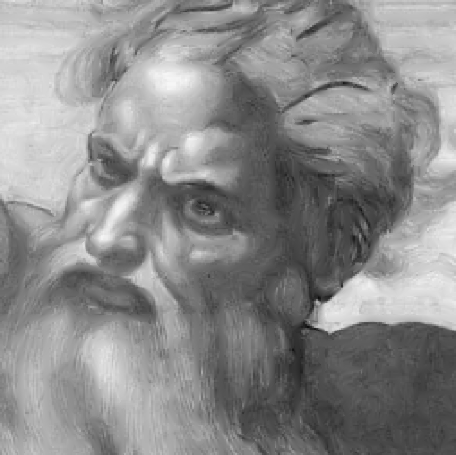}\\
\caption{Left: an occluded image (courtesy of U. Boscain (\cite{Boscain})), second image from left: the image processed in (\cite{Boscain}); third image from left: the same image processed through the heat equation right: image the inpainted using the Citti Sarti algorithm.}
\label{fig:isokernels1}
\end{figure}

\begin{figure}
\centering
\includegraphics[width=0.22\textwidth]{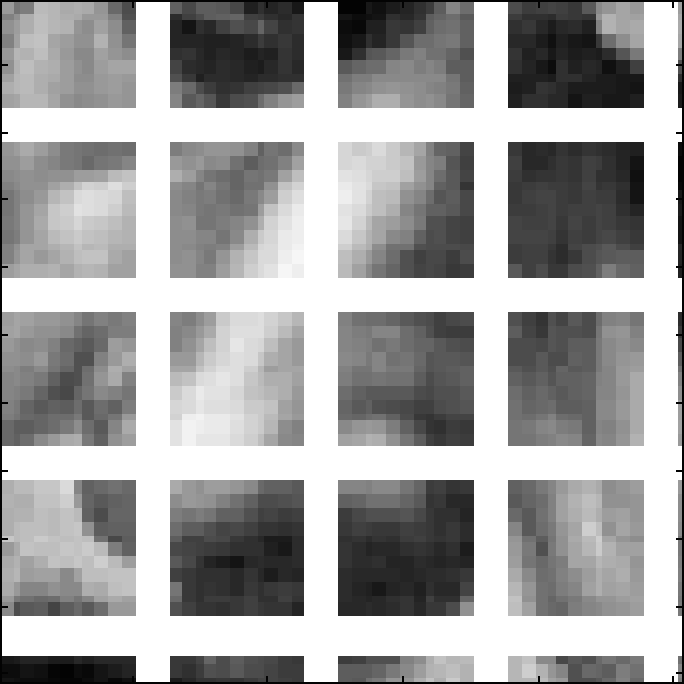}\hspace{.02\textwidth}
\includegraphics[width=0.22\textwidth]{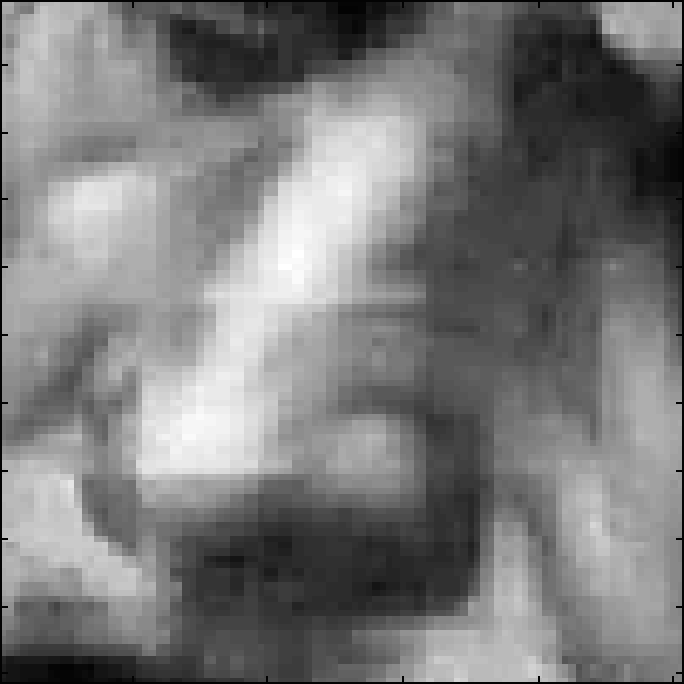}\hspace{.02\textwidth}
\includegraphics[width=0.22\textwidth]{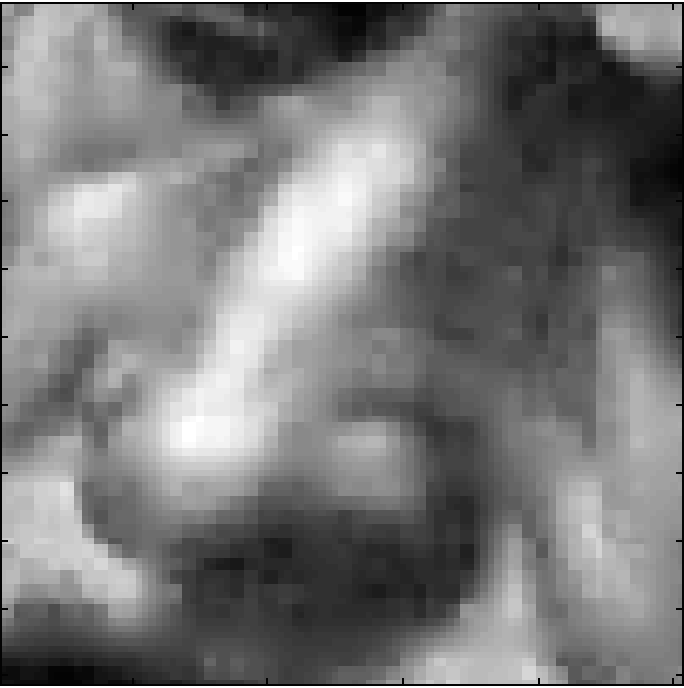}\hspace{.02\textwidth}
\includegraphics[width=0.22\textwidth]{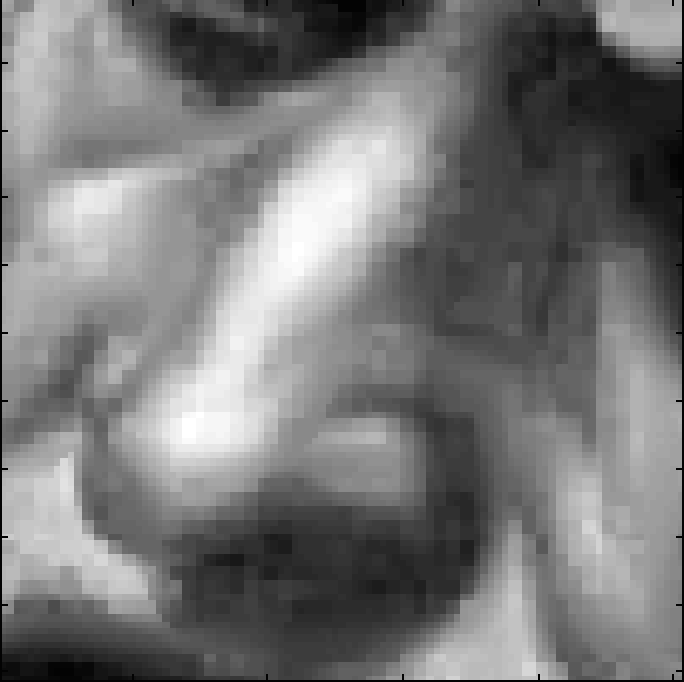}
\caption{A detail of previous image: Left: the original image (courtesy of U. Boscain (\cite{Boscain})); Second image from left: the image processed in (\cite{Boscain}); Third image from left: the image processed through the heat equation;  Right: image inpainted using the proposed algorithm.}
\label{detail1}
\end{figure}

In figure \ref{fig6} (and in the detail taken from it in figure \ref{fig7}) 
we consider an other example taken from the same paper. 
In this image the grid of points which are missed is larger, 
and the previous effect is even more evident.

\begin{figure}
\centering\includegraphics[width=.22\textwidth]{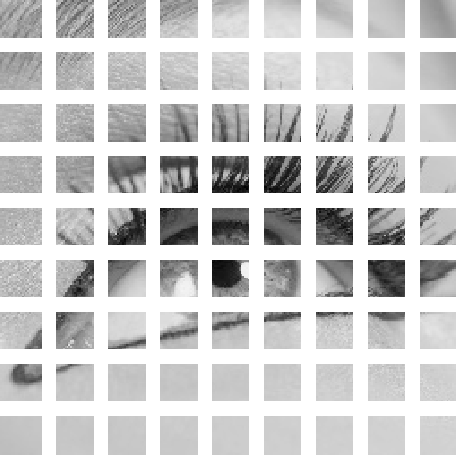}\hspace{.02\textwidth}
\includegraphics[width=.22\textwidth]{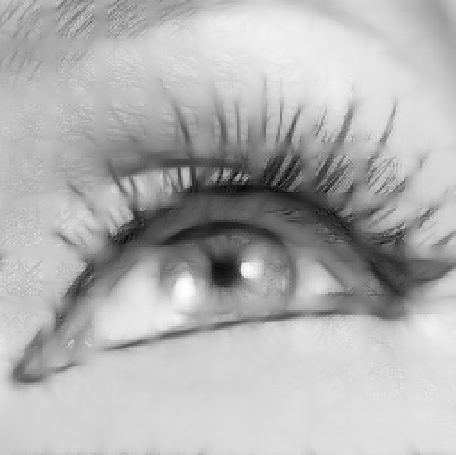}\hspace{.02\textwidth}
\includegraphics[width=.22\textwidth]{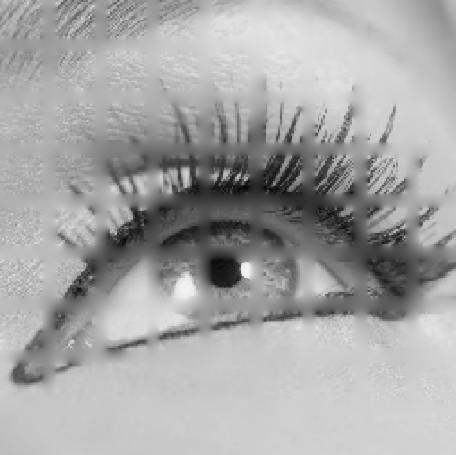}\hspace{.02\textwidth}
\includegraphics[width=.22\textwidth]{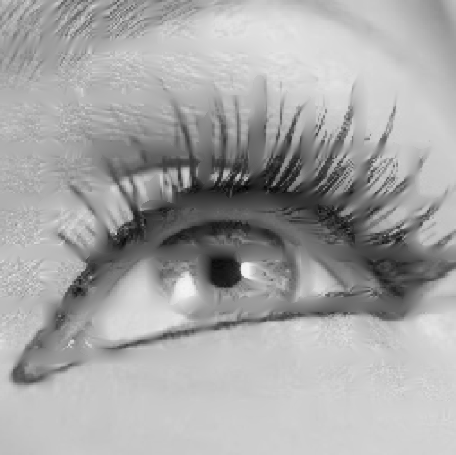}
\caption{Top left: the original image (courtesy of U. Boscain (\cite{Boscain})); Top right: the image processed  in (\cite{Boscain}); Bottom left: the image processed through the heat equation; Bottom right: image inpainted using the proposed algorithm.}
\label{fig6}
\end{figure}

\begin{figure}
\centering
\includegraphics[width=.22\textwidth]{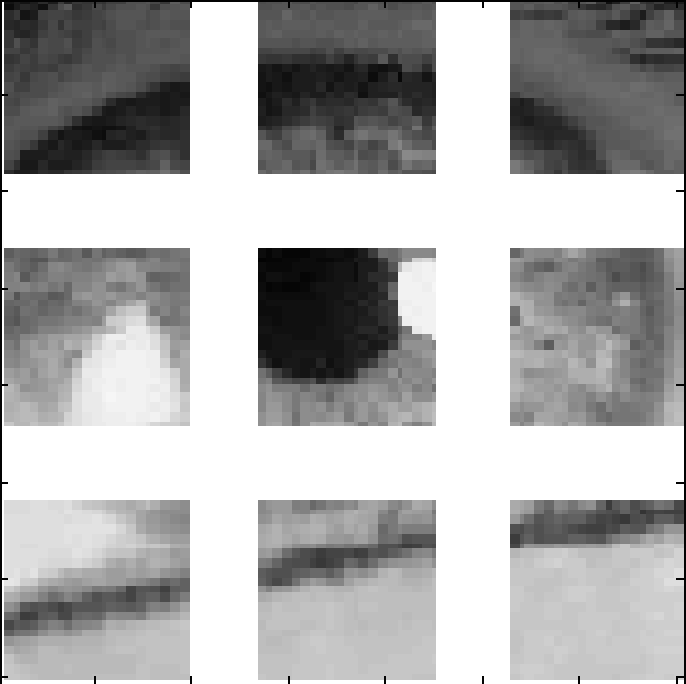}\hspace{.02\textwidth}
\includegraphics[width=.22\textwidth]{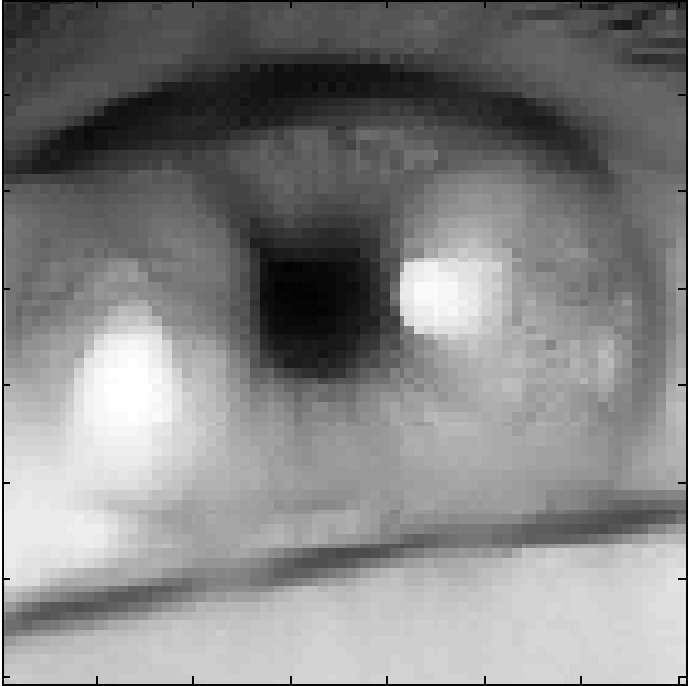}\hspace{.02\textwidth}
\includegraphics[width=.22\textwidth]{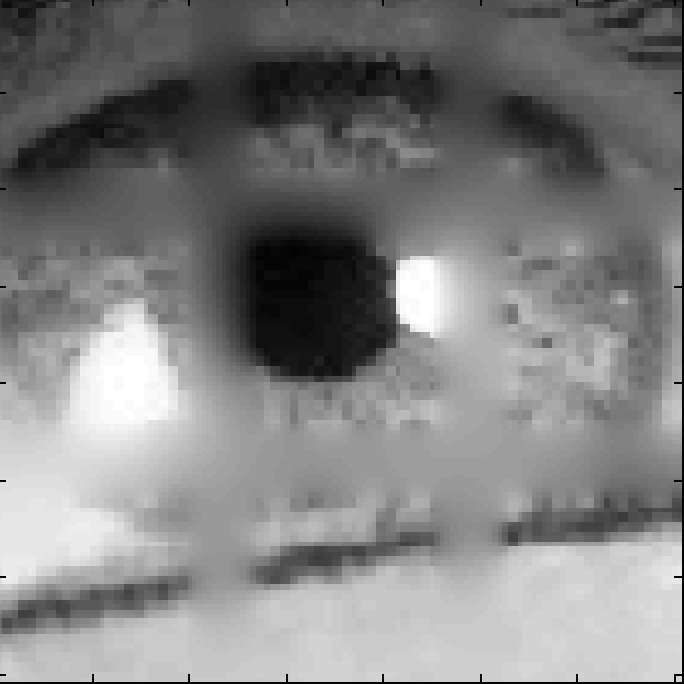}\hspace{.02\textwidth}
\includegraphics[width=.22\textwidth]{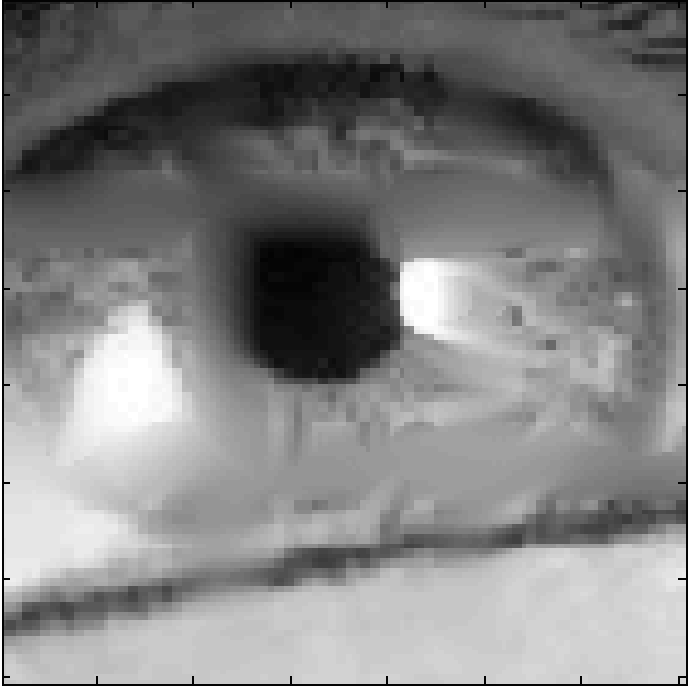}\\
\caption{Detail of the previous image.The result of our algorithm (4th right) preserves the circular level lines of the iris, while simple subRiemannian diffusion (3rd right) destroy it. Simple euclidean diffusion (3rd right)  performs  at intermediate level.}
\label{fig7}
\end{figure}

%
%
%

In a more recent paper Boscain and al. introduced a linear diffusion
with coefficients depending 
on the gradient of the initial image (see \cite{Boscain2}), which they call heuristic. 
In figure \ref{boscainprandi} we compare the results obtained with this model, 
the heat equation on the image plane and the strongly geometric model
of Sarti and Citti.

\begin{figure}
\includegraphics[width=.22\textwidth]{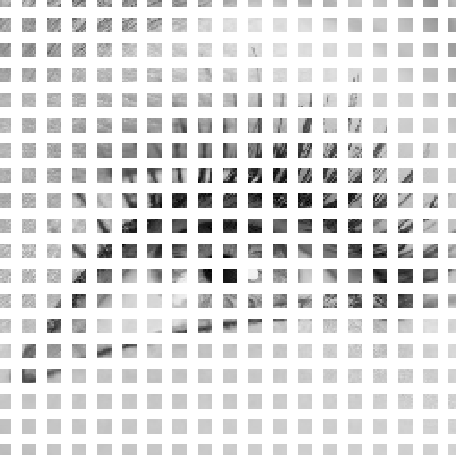}\hspace{.02\textwidth}
\includegraphics[width=.22\textwidth]{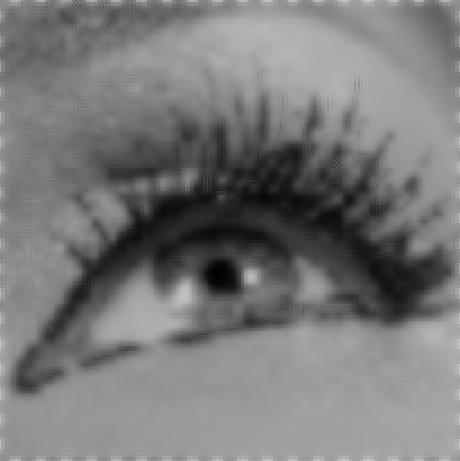}\hspace{.02\textwidth}
\includegraphics[width=.22\textwidth]{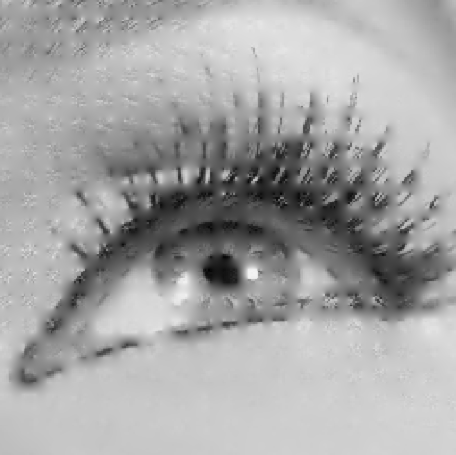}\hspace{.02\textwidth}
\includegraphics[width=.22\textwidth]{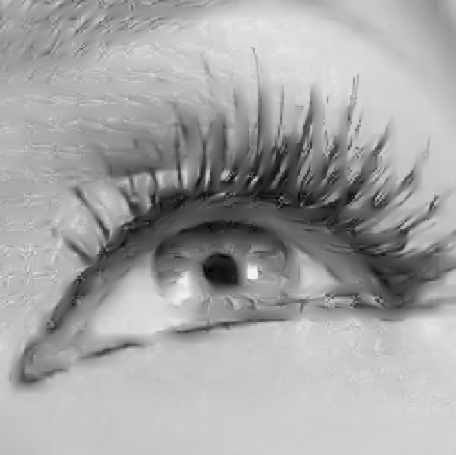}\\
\label{boscainprandi}
\caption{On the left the occluded image. From left to right: results from \cite{Boscain2}, with 2D  heat equation and our model.}
\label{fig8} 
\end{figure}

Then we test our implementation on piecewise constant images. Since the gradient is $0$ in large 
part of the image, the lifted gradient is not defined in largest part of the image. 
On the other side, since the lifting mimics the behavior of the simple cells of the V1 cortical layer, 
the Citti and Sarti algorithm is always applied on a smoothed version of the image.  We have applied it on a classical toy problems proposed for example in Bertalmio, Sapiro, Caselles and Ballester in \cite{Bertalmio}. Results are shown in figure \ref{berta}.

\begin{figure}
\centering
\includegraphics[width=.3\textwidth]{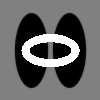}\hspace{.02\textwidth}
\includegraphics[width=.3\textwidth]{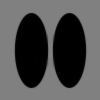}\hspace{.02\textwidth}
\caption{Inpainting a constant coefficient image with the Sarti Citti algorithm.}\label{berta}
\end{figure}

In figure \ref{farfalla} we test our method 
on an image taken from the survey \cite{inpainting}. 
The present reconstruction is correct in the part of the image characterized by strong boundaries, 
but the results of \cite{inpainting} obtained with the model of Morel and Masnou  (see \cite{Morel})  seems to be better. The main point 
is the boundary detection, which is very accurate in the model of Morel and Masnou, while here the boundaries are detected with a gradient, after smoothing the image. 
 
\begin{figure}
\centering
\includegraphics[width=.3\textwidth]{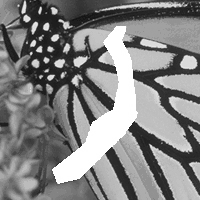}\hspace{.02\textwidth}
\includegraphics[width=.3\textwidth]{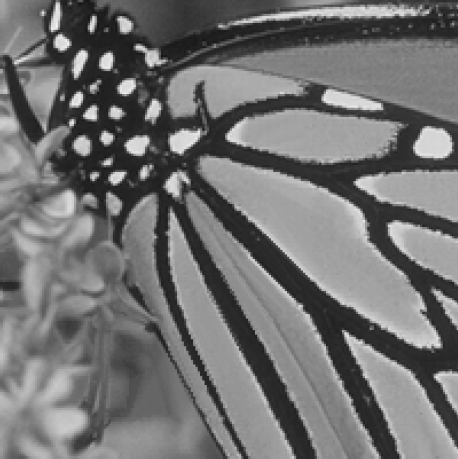}\hspace{.02\textwidth}
\includegraphics[width=.3\textwidth]{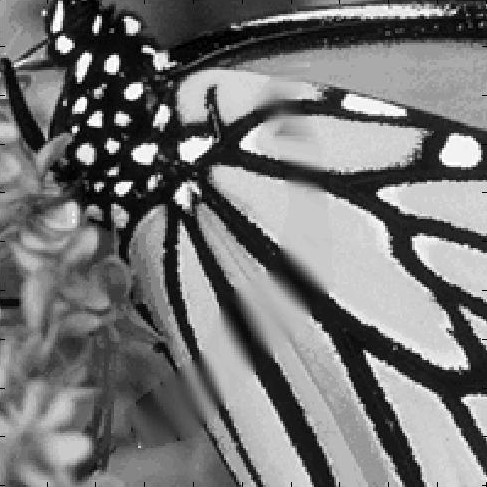}
\caption{On the left the occluded image. From left to right: results from \cite{inpainting} with the model of \cite{Morel}, and with our model.}\label{farfalla}
\end{figure} 

%

\subsection{Enhancement results}
We will show in this section results of the application of the enhancement method we have introduced in Section 2.2.2. Let's recall that enhancement consists in an image filtering that underlines directional coherent structures. With respect to the completion problem there is no part of the image to be disoccluded and all the parts of the initial data are evolved. 

In Figure \ref{fig:isokernels11} it is shown a medical image of blood vessels to be filtered to reconstruct the fragmented vessels (courtesy of R. Duits (\cite{DFII}). The second image from left shows the enhancement computed by using CED-OS, see \cite{DFII}, while the third image shows the result obtained using the proposed method.

\begin{figure}
\centering
\includegraphics[width=.3\textwidth]{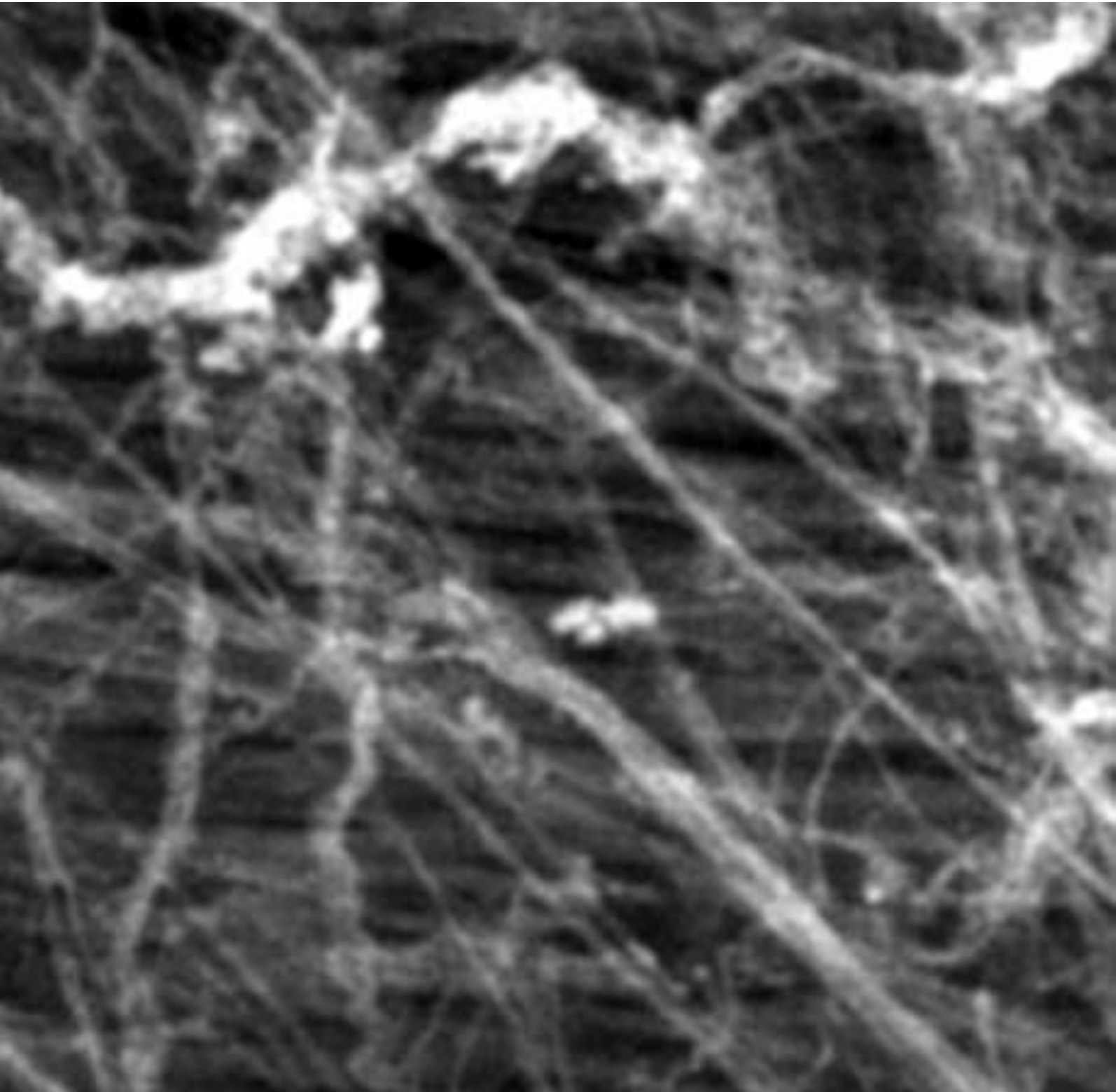}\hspace{.02\textwidth}
\includegraphics[width=.3\textwidth]{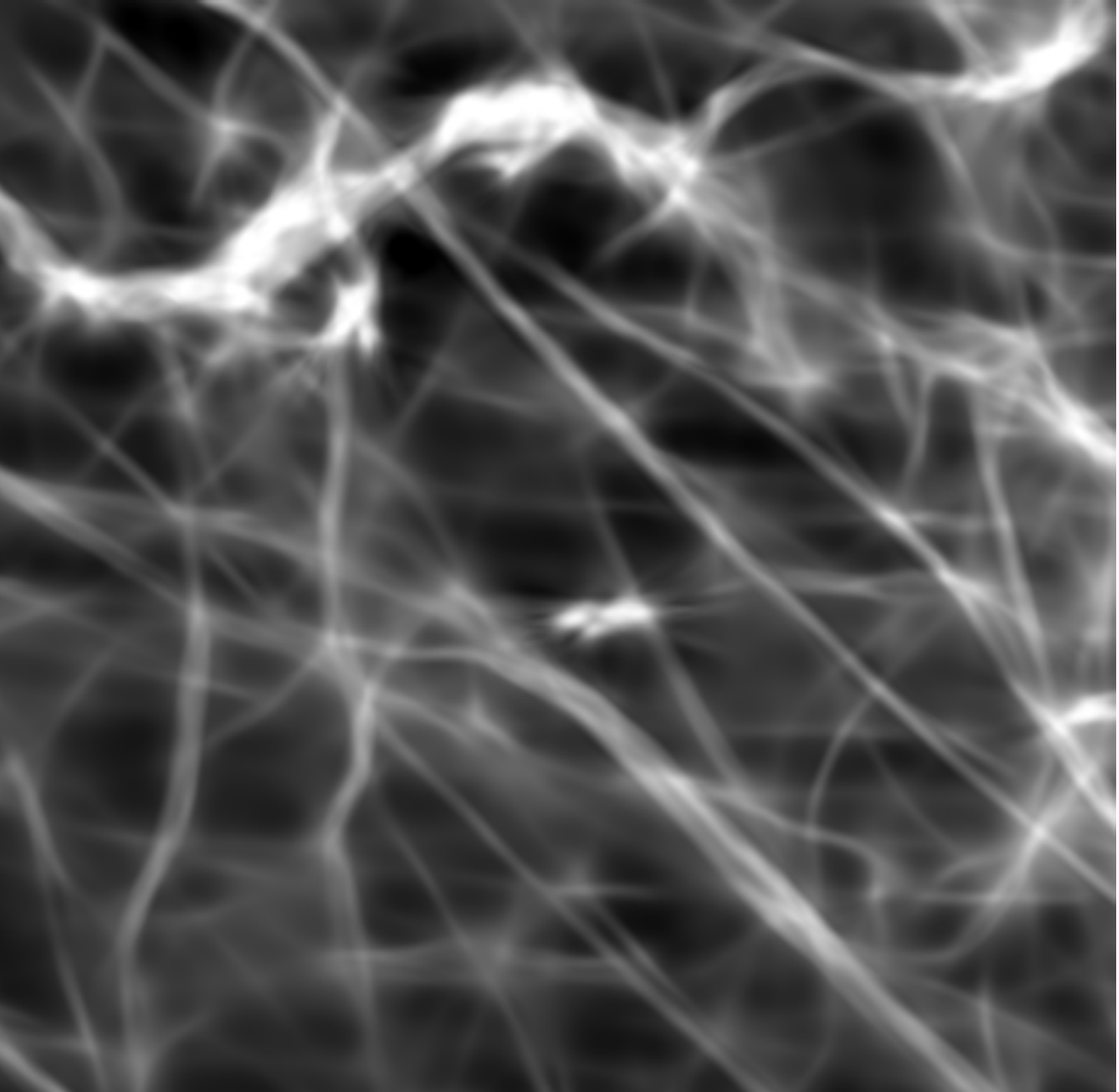}\hspace{.02\textwidth}
\includegraphics[width=.3\textwidth]{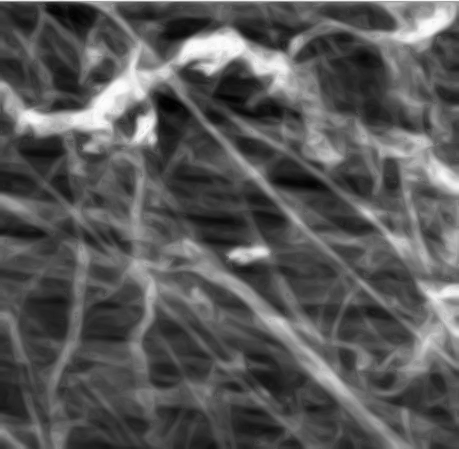}
\caption{From left to right: the original image, courtesy of R. Duits (\cite{DFII}), the enhanced image using CED-OS, see \cite{DFII} and the enhanced image obtained using the proposed method.}
\label{fig:isokernels11}
\end{figure}


Here we finally show an example of  combination of the techniques of completion and enhancement. We see in this case that enhancement homogenizes the original non occluded part with the reconstructed one.
\begin{figure}
\centering
\includegraphics[width=.3\textwidth]{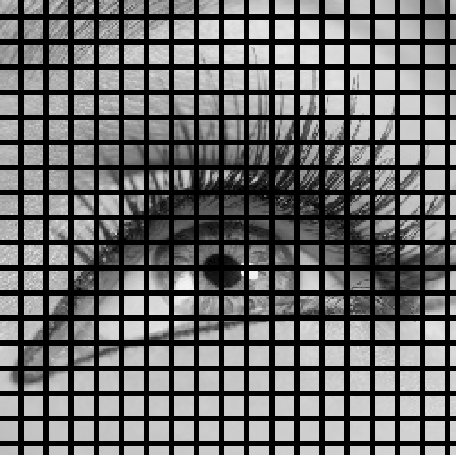}\hspace{.02\textwidth}
\includegraphics[width=.3\textwidth]{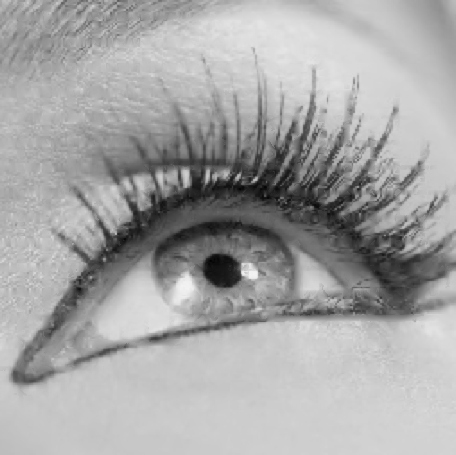}\hspace{.02\textwidth}
\includegraphics[width=.3\textwidth]{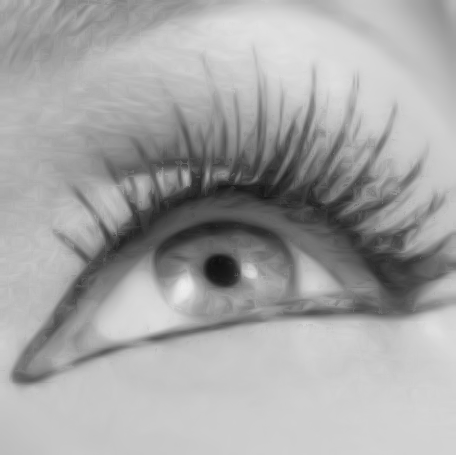}
\caption{Left: the original image (courtesy of U. Boscain (\cite{Boscain})); center: image inpainted using the proposed algorithm; right: image  inpainted and enhanced with this algorithm.}
\label{fig:isokernels21}
\end{figure}
Here we propose a detail of the previous image in order to underline the effects of the discussed techniques.
\begin{figure}
\centering
\includegraphics[width=.3\textwidth]{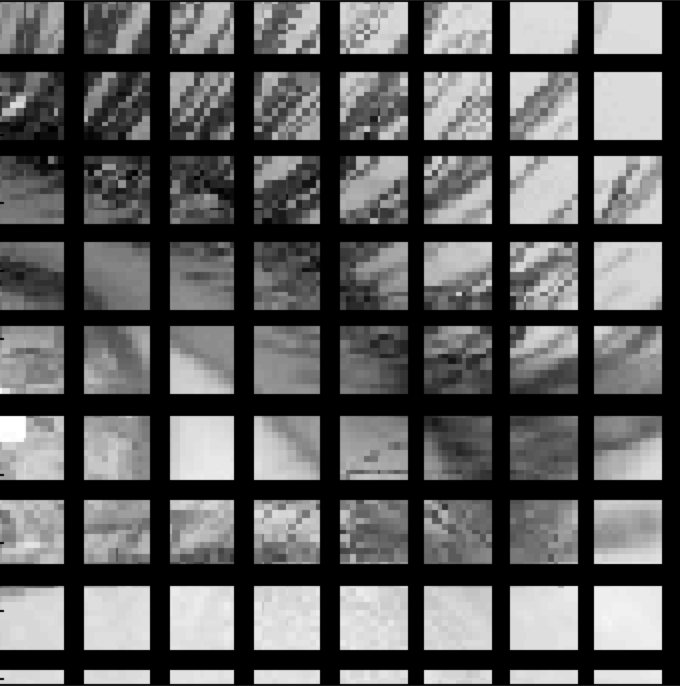}\hspace{.02\textwidth}
\includegraphics[width=.3\textwidth]{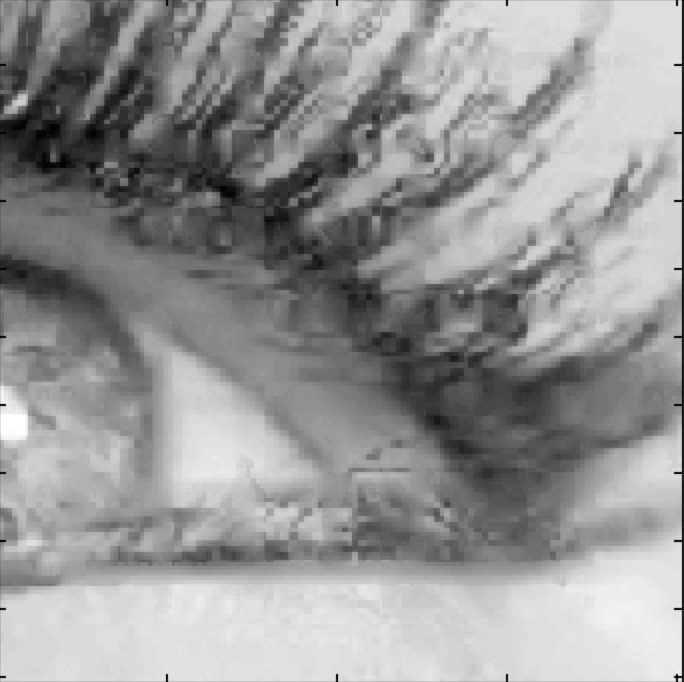}\hspace{.02\textwidth}
\includegraphics[width=.3\textwidth]{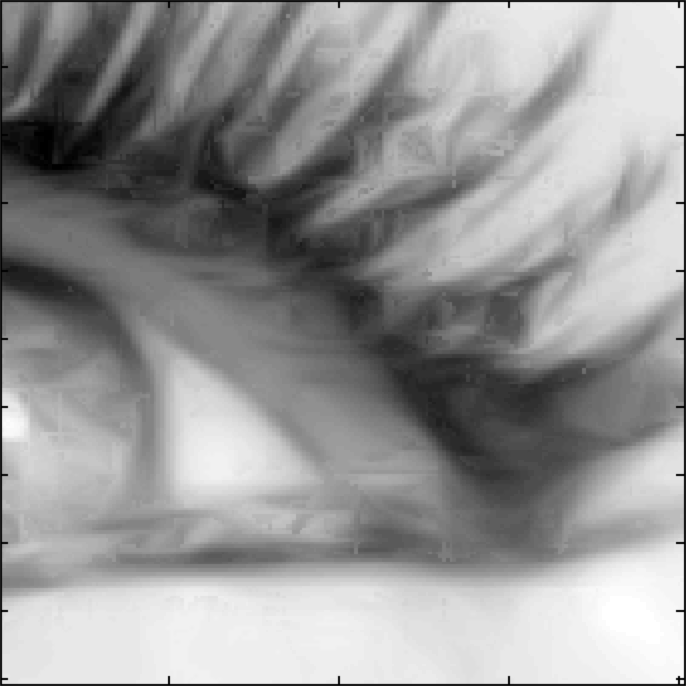}
\caption{Top left: a detail of the original image (courtesy of U. Boscain (\cite{Boscain})); Bottom left: a detail of the image inpainted using the proposed algorithm; Bottom right: same detail of the image inpainted and enhanced with this algorithm.}
\label{fig:detail}
\end{figure}
\section{Conclusions}
In this paper we have proved existence of viscosity solutions of the mean curvature flow PDE in $\mathbb{R}^2 \times \mathit{S}^1$ with a sub-Riemannian metric. The flow has been approximated with the Osher and Sethian technique and a sketch of the proof of convergence of the numerical scheme is provided. Results of completion and enhancing are obtained both on artificial and natural images. We also provide comparisons with others existing algorithms. In particular 
we have illustrated how the method can be used to perform enhancement and how it leads to results comparable with the classical ones of Bertalmio, Sapiro, Caselles and Ballester in \cite{Bertalmio}, of Masnou and Morel in \cite{Morel}.  
In the case of image completion we compared the technique with the recent results shown of Boscain, Chertovskih, Gauthier, Remizov in \cite{Boscain}. Furthermore the method can be applied not only to inpainting problems but also in presence of crossing edges, hence a comparison with the results of edges which cannot be done using the method proposed by Bertalmio, Sapiro, Caselles and Ballester in \cite{Bertalmio} is now possible.